\documentclass[12pt]{article}
\usepackage{amsmath}
\usepackage{amssymb}
\usepackage{theorem}
\usepackage{enumerate} 
\usepackage{hyperref}

\usepackage{bm}

\numberwithin{equation}{section}

\newtheorem{thm}{Theorem}[section]
\newtheorem{lemma}[thm]{Lemma}

\newtheorem{prop}[thm]{Proposition}
\newtheorem{cor}[thm]{Corollary}
{\theorembodyfont{\rmfamily}

\newtheorem{rmk}[thm]{Remark}
}

\newcommand{\essinf}{\operatorname{essinf}}
\newcommand{\qed}{\hfill \mbox{\raggedright \rule{.07in}{.1in}}}
 
\newenvironment{proof}{\vspace{1ex}\noindent{\bf
Proof}\hspace{0.5em}}{\hfill\qed\vspace{1ex}}
\newenvironment{pfof}[1]{\vspace{1ex}\noindent{\bf Proof of
#1}\hspace{0.5em}}{\hfill\qed\vspace{1ex}}

\newcommand{\eps}{{\epsilon}}

\newcommand{\M}{\mathcal{M}}
\newcommand{\R}{\mathbb{R}}
\newcommand{\Z}{\mathbb{Z}}
\newcommand{\N}{\mathbb{N}}

\newcommand{\cB}{\mathcal{B}}

\newcommand{\F}{F}

\title{H\"older continuity of measures for heavy tail potentials}

 \author{Godofredo Iommi
 \thanks{Facultad de Matem\'aticas,
Pontificia Universidad Cat\'olica de Chile (UC), Avenida Vicu\~na Mackenna 4860, Santiago, Chile, godofredo.iommi@gmail.com.}
\and
Dalia Terhesiu
\thanks{Mathematisch Instituut,
University of Leiden, Niels Bohrweg 1, 2333 CA Leiden, Netherlands, daliaterhesiu@gmail.com.
}
\and
{Mike Todd}
\thanks{Mathematical Institute,
University of St Andrews,
North Haugh,
St Andrews,
KY16 9SS,
Scotland, m.todd@st-andrews.ac.uk.
}
}

\date{\today}

\begin{document}

\maketitle


\begin{abstract}
   For a class of potentials $\psi$ satisfying a condition depending on the roof function of a suspension (semi)flow,
 we show an EKP inequality, which can be interpreted as a H\"older continuity property in the weak${^*}$ norm of measures, with respect to the pressure of those measures,
 where the H\"older exponent depends on the $L^q$-space that $\psi$ belongs to. This also captures a new type of phase transition for intermittent (semi)flows (and maps).
\end{abstract}

\section{Introduction}

There exist a wide range of dynamical systems having a unique measure of maximal entropy. That is, there exists a unique measure $\mu_0$ satisfying $h(\mu_0)= \sup\{h(\mu) : \mu \in \M\}$, where $h(\mu)$ denotes the entropy of the measure $\mu$ and $\M$ the space of invariant probability measures. 
If the phase space is compact and  the entropy map is upper semi-continuous (with respect to the weak* topology), if  $(\mu_n)_n$ is a sequence in $\M$ such that $\lim_{n \to \infty} h(\mu_n)=h(\mu_0)$, then $(\mu_n)_n$ 
converges to $\mu_0$. In particular, for any Lipschitz  function $\psi$, we have
$\int \psi \, d \mu_n \rightarrow \int \psi \, d \mu_0$. Polo  \cite[Theorem 4.1.1]{Po11} made this statement effective for  hyperbolic automorphisms of the tori and its corresponding measure of maximal entropy  $\mu_0$ (the Haar measure in case of linear automorphism). Indeed, he proved that there exists a constant $C>0$ such that for any invariant probability measure $\mu$ and any Lipschitz function $\psi$, with Lipschitz constant $L$, 
\begin{equation} \label{eq:polo}
\left|	\int \psi \, d \mu - \int \psi \, d \mu_0	\right| \leq  CL \left(h(\mu_0)-h(\mu)	\right)^{1/3}.
\end{equation}
This result can be thought of as a H\"older continuity property in the weak* norm of measures.  According to Polo \cite[p.6]{Po11}  it was Einsiedler who outlined the argument for the proof of equation \eqref{eq:polo} in the case of $\times 2$ map. Kadyrov \cite[Theorem 1.1]{Kad15} later extended this result to sub-shifts of finite type (defined over finite alphabets). 
In his case, instead of a cubic root, he had a quadratic root.  
Inequalities as \eqref{eq:polo} are now called EKP-inequalities after these authors. 
The case of countable Markov shifts has been studied recently. In that setting the phase space is no longer compact 
and the entropy map is not always upper semi-continuous. Moreover, there  
are cases in  which there is no measure of maximal entropy. Therefore, further assumptions are required in order for 
EKP-inequalities to make sense. For example, R\"uhr \cite[Theorem 1.1]{Ruh21} studied countable Markov shifts satisfying a combinatorial assumption (the BIP-property). This class of systems shares many properties with sub-shifts of finite type. However, they  have infinite entropy, thus EKP-inequalities do not make sense for the measures of maximal entropy.  
Instead, he considered the Gibbs measure associated to a locally H\"older function of finite pressure. 
In that setting, the right hand side of the EKP-inequality has the free energy of the measures (instead of the entropy) 
and a square root. Since systems having the BIP property 
are similar to subshifts of finite type, the arguments in the proof are close to those developed by Kadyrov. 

Sarig and  
R\"uhr  recently studied finite entropy countable Markov shifts. In this case, instead of making a strong assumption 
on the system, they consider strongly positive recurrent (SPR) functions. 
Potentials in this class have unique equilibrium measures and the corresponding transfer operator acts with spectral 
gap in appropriate Banach spaces \cite[Theorem 2.1]{CyrSar09}. They proved \cite[Theorem 6.1]{RuhSar22} that if $\phi$ 
is an SPR regular function, $\mu_{\phi}$ is the associated equilibrium measure and $\psi$ a regular function, 
then for any invariant measure $\mu$ with sufficiently large free energy $P_{\mu}(\phi)$ 
(see Section~\ref{sec:tf}) we have
\begin{equation} \label{eq:ruh-sar}
\left|	\int \psi \, d \mu - \int \psi \, d \mu_\phi	\right| \leq  C\sqrt2 \sigma \sqrt{P(\phi)- P_{\mu}(\phi)},
\end{equation}
where $P(\phi)$ is the pressure of $\phi$ and $\sigma^2$ is the asymptotic variance of $\psi$ with respect to $\mu_\phi$ 
(which in turn is related to the second derivative of the pressure function) and $C$ is a constant which can be taken close to 1 provided $\left|	\int \psi \, d \mu - \int \psi \, d \mu_\phi	\right|$ is small. They also provide a version where the integrals can be far apart, and where $C \sigma$ is replaced by $C'\|\psi\|_\beta$ for a suitable norm, where $C'$ is independent of $\psi$.

In this article we prove EKP-inequalities for continuous time dynamical systems which may not be SPR and can have unbounded entropy, for some unbounded $\psi$. Indeed, we study suspension (semi)flows over Gibbs Markov maps  $T:Y\to Y$, and unbounded roof function $\tau:Y\to (0, \infty)$
with $\inf\tau>0$ satisfying certain additional assumptions.  Our main focus is towards systems with weak hyperbolicity properties. We denote the (semi)flow by $(\F_t)_t$ and the suspension space by $Y^{\tau}$.   We refer to Section~\ref{sec:setup} for details.  Consider a regular potential $\phi$ and its corresponding positive entropy equilibrium state $\nu_{\phi}$. In our main results we establish several EKP-inequalities for $\nu_{\phi}$, for a regular function $\psi$ and for invariant measures $\nu$ satisfying $\int \psi \, d \nu > \int \psi \, d \nu_{\phi}$. We bound the difference $\int \psi \, d \nu - \int \psi \, d \nu_{\phi}$ with terms of the form $(P(\phi) - P_{\nu}(\phi))^{\rho}$. The values of $\rho$ are related to dynamical properties of the system. 

In order to be more precise, we have two basic assumptions. The first (GM0) describes the decay of the tail of the measure on the base map $T$. It essentially says that there exists $\beta >1$ such that $\mu(\tau >x) \leq cx^{-\beta}$.
In order to state our second assumption (GM1), recall that every potential $\psi$ for the (semi)flow has an induced version $\bar \psi$ defined on $Y$. The assumptions of our results are in terms of the induced potentials.  It states that $\bar \psi= C_0 -\psi_0$, where $0 \leq \psi_0 \leq C_1 \tau^{\gamma}$ and $\gamma \in (\beta-1, \beta)$. 
We stress that these assumptions are fulfilled by a wide range of functions $\psi$.   

In our first result, Theorem~\ref{thm:firstmain}, we assume that $\beta / \gamma >3$. 
We show that there exists $\epsilon>0$ such that for any flow invariant probability measure $\nu$, with  
$\int\psi~d\nu\in \left(\int\psi~d\nu_\phi, \int\psi~d\nu_\phi+\eps\right)$, we have
 \begin{equation*}
 \int\psi\,d\nu-\int\psi\,d\nu_{\phi} 
     \le C_{\phi, \psi}  {\sqrt 2}\sigma\sqrt{P_{\nu_\phi}(\phi)-P_\nu(\phi)},
 \end{equation*}
 where $\sigma^2$ is the asymptotic variance of $\psi$ with respect to $\nu_{\phi}$ and where $C_{\phi, \psi} \geq 1$ tends to $1$ as $\int \psi \, d\nu \to \int \psi \, d\nu_{\phi}$.

We note that in the expression above, as well as those in (a) and (b) below, are only useful when $\int\psi\,d\nu>\int\psi\,d\nu_{\phi}$.  It can be shown in the main examples of this theory that this is intrinsically necessary\footnote{though if $\mu(\tau >x)$ decays exponentially then the proofs can be rewritten to recover a statement like \eqref{eq:ruh-sar}.}, rather than an artefact of the proof,  i.e., we can not put absolute value signs on the left hand side of these equations and allow $\int\psi\,d\nu<\int\psi\,d\nu_{\phi}$, see Remark~\ref {rmk:EKPfail}.

In our second main result, Theorem~\ref{thm:secmain}, we consider the cases in which $\beta/ \gamma \in (1,2]$ 
and $\beta / \gamma \in (2,3)$ (with some additional assumptions). This result captures a new type of phase transition. 
Indeed,  while item (b)  below shows a EKP inequality in the case $\beta/\gamma\in (2,3)$ (when the Central Limit Theorem (CLT) is present), 
item (a) gives a new type of EKP inequality with the exponent changing from $1/2$ to one depending on the 
ratio $\beta/\gamma$. Interestingly, this result captures the transition form stable law to CLT in terms 
of the H\"older continuity of the pressure (see Remark~\ref{rmk:st}.)

\begin{itemize}
 \item[(a)]  If $\beta/\gamma\in (1,2]$,
then
\[
 \int\psi\,d\nu-\int\psi\,d\nu_{\phi}\le c_2 (P_{\nu_\phi}(\phi)-P_\nu(\phi))^{\frac{\beta-\gamma}{\beta-\gamma+1}}.
 \]

\item[(b)] If $\beta/\gamma\in (2,3)$,  then  
\[
 \int\psi\,d\nu-\int\psi\,d\nu_{\phi}\le  c_3 {\sqrt 2}\sigma\sqrt{P_{\nu_\phi}(\phi)-P_\nu(\phi)}.
 \]
\end{itemize}

The above results give the most interesting behaviour, and best constants, when $ \int\psi\,d\nu$ and $\int\psi\,d\nu_{\phi}$ are close to each other, but we also give a result Theorem~\ref{thm:bigint} similar to the above when these integrals are far away from each other.

The proof of our results is based on asymptotic estimates of the pressure function $s \mapsto P_F(\phi + s\psi)$. 
For example, in Proposition~\ref{prop:gm} we prove, under the assumptions (GM0) and (GM1), that if $q_1 \in [1, \beta / \gamma]$ then $P_F(\phi + s\psi)$ is of class $C^{q_1}$. 
In Proposition~\ref{prop:gmsec}, under (GM0)  and an assumption on the decay of the tail of the measure, we establish estimates of the type:   if $\beta/\gamma\in (1,2]$, then $ P_F''(\phi + s\psi)=C s^{\beta-\gamma-1}(1+o(1))$. Moreover, if  $\beta/\gamma\in (2, 3)$, then $P_F''(\phi + s\psi)=-C s^{\beta-2\gamma-1}(1+o(1))$. These estimates are essential in the proofs of the main results and are obtained building up from~\cite{BruTerTod19, BruTerTod21, MelTer17}. 
With these in hand, we make use of the restricted pressure in a similar way to \cite{RuhSar22}.

In Section~\ref{sec:examples}, examples of dynamical systems for which the results obtained in the article apply 
are provided. We construct suspension flows over maps exhibiting weak forms of hyperbolicity. 
Indeed, the class of interval maps we consider have parabolic fixed points. This shows the strength of our main results.\\

\textit{Acknowledgements.} 
GI was partially supported by Proyecto Fondecyt 1230100.
DT would like to thank Henk Bruin for discussions on related topics during the
Research-in-Teams project
0223 ``Limit Theorems for Parabolic Dynamical Systems''
at the Erwin Schr\"odinger Institute, Vienna.
MT would like to thank Pontificia Universidad Cat\'olica de Chile where part of this research was done, supported by Proyecto Fondecyt 1230100, and thanks the University of Leiden for hosting a visit where part of this research was done.  He is also partially supported by the FCT (Fundação para a Ciência e a Tecnologia) project 2022.07167.PTDC. We thank the referee for their comments, in particular for a question which led to the inclusion of Theorem~\ref{thm:bigint}.

\section{Suspension flows over Gibbs Markov (GM) maps with unbounded roof $\tau$}
\label{sec:setup}

\subsection{Thermodynamic formalism for suspension flows} \label{sec:tf}

 Let $T:Y\to Y$ be a map and $\tau:Y\to (0, \infty)$ a positive function with $\inf\tau>0$. Consider the space  $Y^\tau=Y\times [0, \infty)/\sim$ where $(y, \tau(y))\sim (T(y), 0)$. The \emph{suspension (semi)flow} over $T$ with \emph{roof function} $\tau$ is the (semi)flow $(\F_t)_t$ defined by $\F_{t'}(y, t) = (y, t+t')$ for $t'\in [0,\tau(y))$.

Denote by $\M_{\F}$ and respectively $\M_T$ the spaces of $F$-invariant  and $T-$invariant probability measures. There is a close relation between these spaces. Indeed, consider the subset of $\M_T$ for which $\tau$ is integrable. That is,
\begin{equation}
\M_T(\tau):= \left\{ \mu \in \mathcal{M}_{T}: \int \tau \,  d \mu < \infty \right\}.
\end{equation}
Let $m$ denote the one-dimensional Lebesgue measure and  $\mu \in \M_T(\tau)$. It follows directly from classical results 
by Ambrose and Kakutani \cite{AmbKak42} that 
\begin{equation}
\nu= \frac{(\mu \times m)|_{Y^\tau} }{(\mu \times m)(Y^\tau)} = \frac{(\mu \times m)|_{Y^\tau} }{\int\tau~d\mu} \in \M_{F}. \label{eq:AK}
\end{equation}
Actually, under the assumption that $\inf\tau>0$,  equation \eqref{eq:AK} establishes a one-to-one correspondence  between measures in $\M_F$ and measures in $\M_T(\tau)$.  We say that $\mu$ is the \emph{lift} of $\nu$ and that $\nu$ is the \emph{projection} of $\mu$.  In the setting of this article, every measure in $\M_F$ lifts to some measure in $\M_T$.

The entropies of measures  as in equation \eqref{eq:AK} are related. Indeed,  for $\mu\in \M_T$ and $\nu \in \M_F$  denote by  $h_T(\mu)$ and $h_F(\nu)$ the corresponding entropies. Abramov \cite{ab} proved that   $h_F(\nu) = \frac{h_T(\mu)}{\int\tau \, d\mu}$.

It is also possible to relate the integral of a function on the flow to a corresponding one on the base. For $\phi:Y^{\tau} \to \R$, we define its induced version  $\bar\phi(x):Y \to \R$ by   $\bar\phi(x) =  \int_0^{\tau(x)} \phi\circ \F_t(x, 0)\, dt$. Let $\mu\in \M_T$ and $\nu \in \M_F$  be related as in equation \eqref{eq:AK}. Kac's formula establishes the following relation:  $\int\phi \, d\nu= \frac{\int\bar\phi \, d\mu}{\int\tau \,d\mu}$.

Having related the spaces of invariant measures, the corresponding entropies and  integrals, it should come as no surprise that thermodynamic formalism on the flow is related to that on the base.  Given  a regular function $\phi:Y^{\tau}  \to \R$, we define the \emph{pressure} of $\phi$ (with respect to the (semi)flow $F$)  by
\begin{equation*}
P_F(\phi) := \sup\left\{h_F(\nu)+ \int\phi~d\nu: \nu\in \M_F \text{ and } \int\phi~d\nu>-\infty \right\}.
\end{equation*}
It will be convenient to write $P_{F, \nu}(\phi) = h_F(\nu)+ \int\phi~d\nu$ for $\nu\in \M_F$, when this sum makes sense. 
We call $\nu\in \M_F$ an \emph{equilibrium state} for $\phi$ if $P_{F, \nu}(\phi)= P_F(\phi)$, and write $\nu=\nu_\phi$.
Similarly, the \emph{pressure} of $\bar\phi : Y \to \R$  (with respect to the map $T$) is defined by
\begin{equation*}
P_T(\bar\phi) := \sup\left\{h_T(\mu)+ \int\bar\phi~d\mu: \mu\in \M_T \text{ and } \int\bar\phi~d\mu>-\infty \right\}.
\end{equation*}
Again, it will be convenient to write $P_{T,\mu}(\bar\phi) = h_T(\mu)+ \int\bar\phi~d\mu$ for $\mu\in \M_T$, 
when this sum makes sense. We call $\mu\in \M_T$ an \emph{equilibrium state} for $\bar\phi$ if $P_{T, \mu}(\bar\phi)= P_T(\bar\phi)$ 
and write $\mu=\mu_{\bar\phi}$. 

\begin{rmk}
 Note that, under the assumptions we have considered here, Abramov's and Kac's formulas imply that,
\begin{equation*}
P_F(\phi)=\sup\left\{\frac{h_T(\mu)+ \int\bar\phi~d\mu}{\int\tau~d\mu}: \mu\in \M_T(\tau) \text{ and } \int\bar\phi~d\mu>-\infty \right\}.
\end{equation*}
We will assume that $P_F(\phi)=0$ (otherwise we can shift the potential by a constant). This implies that $P_T(\bar\phi)\le 0$. Moreover, in this paper liftability of all measures implies in fact that  $P_T(\bar\phi)=0$. Under an integrability condition equilibrium states for $\phi$ and $\bar\phi$ are also related. Indeed, if $\mu_{\bar\phi}\in \M_T(\tau)$ then the equilibrium state for $\phi$ is
\begin{equation*}
\nu_\phi= \frac{(\mu_{\bar\phi} \times m)|_{Y^\tau} }{\int\tau~d\mu_{\bar\phi}}.
\end{equation*}
\end{rmk}

We conclude this section  with the following definition which is analogous to \cite[Definition 3.1]{RuhSar22}:
\begin{equation*}
p(t) := P_F(\phi+t\psi) = \sup\left\{\frac{P_{T, \mu}(\overline{\phi+t\psi})}{\int\tau~d\mu}:\mu\in \M_T(\tau) \text{ and } \int\overline{\phi+t\psi}~d\mu>-\infty\right\}.
\end{equation*}

\subsection{Gibbs Markov maps and the main assumptions}

Roughly speaking, Gibbs-Markov maps are infinite branch uniformly expanding maps with bounded distortion and big images. 
We recall the definitions in more detail.
Let $(Y,\mu_Y)$ be a probability space, and let $T:Y\to Y$ be a topologically mixing ergodic measure-preserving transformation, piecewise continuous w.r.t.\ 
a non-trivial countable partition $\{a\}$.
Define $s(y,y')$ to be the least integer $n\ge0$ such that $T^ny$ and $T^ny'$ lie in distinct partition elements.
Assuming  that $s(y,y')=\infty$ if and only if $y=y'$ one obtains that $d_\theta(y,y')=\theta^{s(y,y')}$
for $\theta\in(0,1)$ is a metric.

Let $g=\frac{d\mu_Y}{d\mu_Y\circ T}:Y\to\R$.
We say that $T$ is a \emph{Gibbs-Markov map} if the following hold w.r.t.\ the countable partition $\{a\}$:
\begin{itemize}

\parskip = -2pt
\item $T|_a:a\to T(a)$ is a measurable bijection for each $a$
such that $T(a)$ is the union of elements of the partition $\bmod\mu_Y$;
\item $\inf_a\mu_Y(T(a))>0$ (the big image property);
\item
There are constants $C>0$, $\theta\in(0,1)$ such that
$|\log g(y)-\log g(y')|\le Cd_\theta(y,y')$ for all $y,y'\in a$ and for all $a\in\{a\}$.
\end{itemize}
See, for instance,~\cite[Chapter 4]{Aaronson} and~\cite{AD01}
for background on Gibbs-Markov maps.  Note that under these assumptions, since our system can be viewed as a topologically mixing countable Markov shift with $\mu_Y$ as an equilibrium state for $\log g$, $\mu_Y$ must have positive entropy, see for example \cite[Theorem 5.6]{Sar15}.

Given $v:Y\to\R$, let
\[
D_av=\sup_{y,y'\in a,\,y\neq y'}|v(y)-v(y')|/d_\theta(y,y'),\qquad |v|_\theta=\sup_{a\in\{a\}}D_a v.
\]
The space $\cB_\theta\subset L^\infty(\mu_Y)$ consisting of the functions
$v:Y\to\R$ such that $|v|_\theta<\infty$ with norm
$\|v\|_{\cB_\theta}=|v|_\infty+|v|_\theta<\infty$ is a Banach space. It is known that the transfer operator $R: L^1(\mu_Y)\to L^1(\mu_Y)$, $\int_Y R v w\, d\mu_Y=\int_Y v w\circ T\, d\mu_Y$ has a spectral gap in $\cB_\theta$ (see,~\cite[Chapter 4]{Aaronson}). In particular, this means that $1$ is a simple eigenvalue, isolated in
the spectrum of $R$.

We will also be interested in functions $v:Y\to \R$ such that there is some $C>0$ so that
 \begin{align}\label{eq:inf}
 D_av\le C\inf(1_av), \qquad \forall a\in\{a\}.
 \end{align}

To connect the measures preserved by Gibbs-Markov maps to the previous section we will assume that $\log g = \bar\phi$, so that $\mu_Y=\mu_{\bar\phi}$ is the equilibrium state for $\bar\phi$.  We will use this notation interchangeably.  As in the previous section, under our assumptions, $\mu_{\bar\phi}$ will project to $\nu_\phi$, the equilibrium state for $\phi$.

In this section, we assume that the roof function $\tau:Y\to\R_{+}$
is unbounded and so that
\begin{itemize}
 \item[(GM0)] $\mu_Y(\tau\ge x)\le c x^{-\beta}$, $\beta>1$ for some $c>0$ depending on the map $T$.
 Moreover, we assume that
  $\essinf\tau>0$ ($\essinf$ w.r.t.\ $\mu_Y$)
 and that $\tau$ satisfies \eqref{eq:inf}.
\end{itemize}

The class of potentials we shall work with is as in~\cite{BruTerTod19, BruTerTod21}, which is very natural in the unbounded roof function case.  
Given the suspension $Y^\tau$ and the suspension flow $\F:Y^\tau\to Y^\tau$, consider the potential $\psi: Y^\tau\to \R$.
Our assumptions are in terms of the induced potentials $\overline\psi(x)$.

\begin{itemize}
 \item[(GM1)] Under (GM0), we further assume that $\overline\psi=C_0-\psi_0$, where $0\le\psi_0 (y)
 \le C_1\tau^\gamma (y)$, for $C_0, C_1 > 0$ and
$\gamma\in(\beta-1, \beta)$. Moreover, we assume that
$\essinf\psi_0>0$, $\psi_0$ satisfies~\eqref{eq:inf} and $\int\psi~d\nu_\phi>0$.
\end{itemize}

\begin{rmk}\label{rmk:p>0}
The assumption $\int\psi~d\nu_\phi>0$ in (GM1) ensures that  $p(s)>0$ for $s>0$, which we require throughout. Indeed, $p(s) \ge h_F(\nu_\phi)+\int\phi+s\psi~d\nu_\phi=s \int\psi~d\nu_\phi>0$.  In fact, standard arguments in thermodynamic formalism, see for example \cite[Theorem 4.6.5]{PrzUrb10} and \cite{Sar06} imply that the potentials $\phi+s\psi$ are positive recurrent for $s>0$ and right derivative $D^+p(0)=  \int\psi~d\nu_\phi$.

We can always make $\int\psi~d\nu_\phi$ make positive by replacing $\psi$ by $\psi+c\cdot 1_Y$ for some constant $c$ as in \cite[Remark 8.4]{BruTerTod19}.
The induced potential becomes $\bar\psi+c$ which does not change the tail behaviour, but can make the integral strictly positive.
\end{rmk}

We note that under (GM0), 
\begin{align}\label{tauq}
 \tau\in L^{q_0}(\mu_{\overline{\phi}}),\text{ for any } 1\le q_0<\beta.
\end{align}
and under (GM1),
\begin{align}\label{psi0q}
 \psi_0\in L^{q_1}(\mu_{\overline{\phi}}),\text{ for any } 1\le q_1<\beta/\gamma.
\end{align}
Let $\overline\psi_n=\sum_{j=0}^{n-1}\overline\psi\circ T^j$.
We note that for $q_1>2$ (so, $\beta/\gamma>2$) and 
$\frac{\overline\psi_n-n\mu_{\overline{\phi}}(\bar\psi)}{\sqrt n}$ converges in distribution to a Gaussian random variable with
zero mean and variance $\bar\sigma^2=\lim_{n\to\infty}\frac{1}{n}\int_Y \left(\bar\psi_n-\int_Y \bar\psi_n\,d\mu_{\overline{\phi}}\right)^2\,d\mu_{\overline{\phi}}$.
Because $\overline\psi$ is unbounded, following~\cite[Theorem 3.7]{Gou04},  to ensure that $\bar\sigma^2> 0$ we need to clarify two things. (We recall that $R$ is the transfer operator for $T$
with spectral gap in $\cB_\theta$.) Given  $\overline\psi=C_0-\psi_0$ with $q_1>2$ (so, $\beta/\gamma>2$), let 
 $\Phi=\overline\psi-\int_Y \overline\psi\, d\mu_{\overline{\phi}}$.

We will also require:
 \begin{itemize}
 \item[(a)]$R(\Phi v)\in\cB_\theta$ for all $v\in\cB_\theta$.
 \item[(b)] There exists no function $h\in\cB_\theta$ so that $\Phi=h-h\circ T$.
 \end{itemize}

\begin{rmk}\label{rmk:use1111}
 Item (a) is verified (in the setup of Gibbs Markov maps) inside the proof of Lemma~\ref{lemma:sm} below (see, in particular,~\eqref{eq:use1}). Item (b) simply requires that $\overline\psi$ is not cohomogous to a constant.
 As soon as $\overline\psi$ (equivalently $\psi$) is not cohomogous to a constant, formula~\eqref{eq:sigma} below ensures that $\sigma^2>0$.
\end{rmk}

A classical lifting scheme~\cite{MT04}
ensures that the CLT holds for the original potential
$\psi:Y^\tau\to Y^\tau$ with mean zero and non zero variance $\sigma^2$.
In this case, given that 
$\nu_{\phi}=\frac{\mu_{\overline{\phi}}\times m|_{Y^\tau}}{\int_Y\tau\,d\mu_{\overline{\phi}}}$ 
is the unique equilibrium state for $\phi$ (this is a classical lifting scheme: see, for  instance, 
the review in~\cite[Section 3]{BruTerTod21}). 
Let 
\begin{equation*}
 \sigma^2 = \lim_{T\to\infty}\frac{1}{T}Var(\psi_T),\quad \psi_T=\int_{0}^T\psi\circ \F_t\, dt,\quad Var(\psi_T)=\int_{Y^\tau}\left(\psi_T-\int_{Y^\tau}\psi_T\, d\nu_{\phi}\right)^2\, d\nu_{\phi}.
\end{equation*}
It follows from~\cite{MT04} that, for $\tau^* := \int_Y \tau \, d\mu_{\bar \phi}$,
\begin{equation}\label{eq:sigma}
 \sigma^2=\frac{\bar\sigma^2}{\tau^*},
 \text{ where }\bar\sigma^2=\lim_{n\to\infty}\frac{1}{n}\int_Y \left(\bar\psi_n-\int_Y \bar\psi_n\,d\mu_{\overline{\phi}}\right)^2\,d\mu_{\overline{\phi}}.
 \end{equation}
 We also write $\sigma_{\nu_\phi}(\psi)^2$ when we wish to emphasise the dependence on $\phi$ and $\psi$.

\subsection{Key propositions}
  Note that in general the derivatives $p'(s), p''(s)$ of our pressure functions are not defined at $s=0$: we will be interested in the derivatives from the right,  but to save notation we will write $p'(0), p''(0)$ and so on, rather than $D^+p(0), (D^2)^+p(0)$.  Similarly for the function $q_{\phi, \psi}$ used later.
Combining and adapting arguments from~\cite{BruTerTod19, BruTerTod21, MelTer17},
we obtain the following result.

\begin{prop}\label{prop:gm}
 Assume (GM0) and (GM1).  Assume that $q_0\in [1,\beta)$ and $q_1\in [1,\beta/\gamma)$.
 Then there exists $\delta_0>0$ so that for all $u,s\in [0,\delta_0)$,
 
 \begin{itemize}
  \item[(i)] 
   $\bar p(u,s):=P_{T}(\overline{\phi+s\psi-u})$ is $C^{q_0}$ in $u$
   and 
   $C^{q_1}$ in $s$.
  \item[(ii)] Define $p(s):=P_{\F}(\phi+s\psi)$. Then
  \begin{align*}
  p(s)=\frac{\bar p(0,s)}{\tau^*}(1+o(1)), \text{ as } s\to 0.
  \end{align*}
  Also $p(s)$ is $C^{q_1}$ and
   $p'(0)=\frac{\overline{\psi}^*}{\tau^*}:=\frac{\int_Y\overline{\psi}\, d\mu_{\bar\phi}}{\int_Y\tau\, d\mu_{\bar\phi}}$.
   
  \item[(iii)] Suppose $q_1>2$ and $p''(0)=\sigma^2$, where $\sigma^2=\sigma_{\nu_\phi}(\psi)^2$ is as in~\eqref{eq:sigma}.
  
  \end{itemize}
\end{prop}

\begin{rmk}\label{rmk:clspot}
 We note that the restrictions posed on the class of potentials considered in (GM1) is not just a matter of simplification. 
Hypothesis (GM1) or variants of it are needed to ensure that the transfer operators perturbed with real valued potentials 
 defined in Section~\ref{sub:firstprop} are well defined in $\cB_\theta$.
 This is a necessary ingredient for the relation between eigenvalues and pressure function: 
 see Section~\ref{sub:firstprop} below.
\end{rmk}

As we will show in Section~\ref{sub:firstprop}, 
 item (ii) of Proposition~\ref{prop:gm} follows from item (i) together with
the Implicit Function Theorem (IFT).
For the case of LSV maps (as in~\cite{LiveraniSaussolVaienti99}; they are a type of AFN maps, see Section~\ref{sec:examples}) with infinite measure, an implicit equation is exploited in the proof of~\cite[Proof of Theorem 4.1]{BruTerTod19}. For the proof of item (i), we adapt the arguments in~\cite{BruTerTod19} to the case of finite measure. For the proof of item (ii), we combine the `implicit' equation
in~\cite[Proof of Theorem 4.1]{BruTerTod19} with the IFT,
which is natural since here we are interested in the smoothness of $P_{T}(\overline{\phi+s\psi})$.

While Proposition~\ref{prop:gm} will allow us to obtain the expected EKP inequality for $q_1>3$
(so $\beta/\gamma > 3$, see \eqref{psi0q}), in the case $\beta/\gamma<3$, we need a refined version
under stronger assumptions. The next proposition tells us how the second derivative of $p(s)$ blows up as $s\to 0$ when  $\beta/\gamma\in (1,2]$
and, how the third derivative blows up as $s\to 0$ when  $\beta/\gamma\in (2,3)$, respectively. (It also gives the speed of convergence of the first and second
derivatives to $p'(0)$ and $p''(0)$, respectively.)

\begin{prop}\label{prop:gmsec} Assume (GM0) with $\mu_Y(\tau\ge x)=c x^{-\beta}(1+o(1))$ for $\beta\in (1,2)$.
Suppose that (GM1) holds with $\psi_0=C_1 \tau^\gamma$ with $\gamma\in (\beta-1,1)$. 
There exist $C_2, C_3>0$ depending only on $c,\beta,\gamma$ and $\tau^*$ so that the following hold as $s\to 0$.
\begin{itemize}
\item[(i)] If $\beta/\gamma\in (1,2]$, then $p''(s)=C_2 s^{\beta-\gamma-1}(1+o(1))$.
\item[(ii)] If $\beta/\gamma\in (2, 3)$, then $p'''(s)=-C_3 s^{\beta-2\gamma-1}(1+o(1))$.
\end{itemize}
\end{prop}
\begin{rmk}
\begin{itemize}
\item[(i)] It is possible to change the assumption on $\beta$ and $\gamma$, but we need a definite assumption to state a final result.
When $\gamma>1$, the asymptotics are different. We do not consider other cases here as this would make the analysis even more tedious,
though most of the calculations can easily be adapted to fit this case.

\item[(ii)]If $\gamma=1$ and $\beta>1$, then we have the following scenario: 
(a) $p''(s)=C_2 s^{\beta-2}(1+o(1))$ if $\beta\in (1,2)$, 
(b)  $p''(s)=C_3 \log(1/s)(1+o(1))$ if $\beta=2$ and
(c) $p'''(s)=C_4 s^{\beta-3}(1+o(1))$
if $\beta\in (2,3)$. We do not display the calculations in this case mainly because it does not lead to any interesting phase transition
in the corresponding version of Theorem~\ref{thm:secmain}.
\end{itemize}
 \end{rmk}

\subsection{Main Theorems}

Using Propositions~\ref{prop:gm} and~\ref{prop:gmsec}, we obtain an interesting generalization
of~\cite{RuhSar22}
for the restricted pressure $q_{\phi, \psi}$. Though our class of potentials is, naturally, much more restricted than in (GM1),
Theorems~\ref{thm:firstmain} and~\ref{thm:secmain}  below show the existence of a new \emph{phase transition}
in terms of whether $\psi_0$ is in $L^{2}(\mu_{\overline{\phi}})$ or not. In particular, 
if $\beta/\gamma>2$ then  $\psi_0$ is $L^2(\mu_{\overline{\phi}})$ (recall~\eqref{psi0q}). The new phase transition is  captured in Theorem~\ref{thm:secmain}.

The result below gives the EKP inequality for $q_1>3$ (with $q_1$ as in~\eqref{psi0q}) when the CLT holds.
Before the statement we note that we are interested in cases $\int\psi~d\nu\neq \int\psi~d\nu_\phi$, so implicitly we are always assuming that $\psi$ is not cohomologous to a constant. 
We also recall from Remark~\ref{rmk:use1111} that this is all what we need to ensure that $\sigma^2>0$.

\begin{thm}\label{thm:firstmain} Assume (GM0) and (GM1).
Assume that $q_1 >3$ (so $\beta/\gamma > 3$) and let $\sigma=\sigma_{\nu_\phi}(\psi)$ be as
defined in~\eqref{eq:sigma}. There exists $\eps>0$ 
  so that for any $F$-invariant probability measure $\nu$ with 
  $\int\psi~d\nu\in \left(\int\psi~d\nu_\phi, \int\psi~d\nu_\phi+\eps\right)$, we have
 \[
\int\psi\,d\nu-\int\psi\,d\nu_{\phi} 
     \le C_{\phi, \psi}  {\sqrt 2}\sigma\sqrt{P_{\nu_\phi}(\phi)-P_\nu(\phi)},
     \] 
where $C_{\phi, \psi} \geq 1$ tends to $1$ as $\int \psi \, d\nu \to \int \psi \, d\nu_{\phi}$.     
     
 For the equilibrium states $\nu_s$ of $\phi + s\psi$, we have
    \begin{equation}\label{eq:ho}
\left| \frac{\int\psi\,d\nu_s-\int\psi\,d\nu_{\phi} }{ \sqrt{P_{\nu_\phi}(\phi)-P_{\nu_s}(\phi)}} - \sqrt 2 \sigma \right| = O\left(\sqrt{P_{\nu_\phi}(\phi)-P_{\nu_s}(\phi)}  \right) \quad \text{ as } s \to 0.
  \end{equation}
\end{thm}

The first result below addresses the case $q_1<3$. We consider two main cases for the ratio $\beta/\gamma$.
It is precisely this result that captures the new type of phase transition. While item (b) of the result below
shows a (familiar) EKP inequality in the case $\beta/\gamma\in (2,3)$ (when the CLT with standard scaling is present),
item (a) gives a new type of EKP inequality with the exponent changing from $1/2$ to one depending on the ratio $\beta/\gamma$.
The transition is natural (see Remark~\ref{rmk:st}).

\begin{thm}\label{thm:secmain}Assume (GM0) with $\mu_Y(\tau\ge x)=c x^{-\beta}(1+o(1))$, with $\beta\in (1,2)$.
Suppose that (GM1) holds with $\psi_0=C_1 \tau^\gamma$ with $\gamma\in (\beta-1,1)$. 

There exist $\eps>0$ and constants $c_2, c_3>0$ so that the following hold for any $F$-invariant probability measure $\nu$ with $\int\psi~d\nu\in \left(\int\psi~d\nu_\phi, \int\psi~d\nu_\phi+\eps\right)$.
\begin{itemize}
 \item[(a)]  If $\beta/\gamma\in (1,2]$,
then
\[
 \int\psi\,d\nu-\int\psi\,d\nu_{\phi}\le c_2 (P_{\nu_\phi}(\phi)-P_\nu(\phi))^{\frac{\beta-\gamma}{\beta-\gamma+1}}
 \]
 For the equilibrium states $\nu_s$ of $\phi + s\psi$, there is a constant\footnote{for $C_2$ as in Proposition~\ref{prop:gmsec}(i).}
 $C_2>0$ such that
 \begin{equation}
\left| \frac{ \int\psi\,d\nu_s-\int\psi\,d\nu_{\phi}}{(P_{\nu_\phi}(\phi)-P_{\nu_s}(\phi))^{\frac{\beta-\gamma}{\beta-\gamma+1}} }
 - \frac{\beta}{\gamma} C_2^{-\frac{1}{\beta-\gamma}} \right| = o(1) \quad \text{ as } s \to 0.
 \label{eq:heavyHold}
 \end{equation}

\item[(b)] If $\beta/\gamma\in (2,3)$, then
\[
 \int\psi\,d\nu-\int\psi\,d\nu_{\phi}\le  c_3 {\sqrt 2}\sigma\sqrt{P_{\nu_\phi}(\phi)-P_\nu(\phi)}.
 \]
 For the equilibrium states $\nu_s$ of $\phi + s\psi$, we have
 \begin{equation}
\left| \frac{ \int\psi\,d\nu_s- \int\psi\,d\nu_{\phi}}{\sqrt{P_{\nu_\phi}(\phi)-P_{\nu_s}(\phi)}} - \sqrt2\sigma \right|
 = O\left((P_{\nu_\phi}(\phi)-P_{\nu_s}(\phi) )^{\frac{\beta-2\gamma}{2}}\right) \quad \text{ as } s \to 0.
 \label{eq:heavyquad}
 \end{equation}
\end{itemize}
\end{thm}
\begin{rmk}\label{rmk:nonzpt}
 We note that $P_{\nu_\phi}(\phi)-P_{\nu}(\phi)$ in the theorems above cannot be zero because $\nu_\phi\neq \nu$ and $\nu_\phi$ is the unique equilibrium state for $\phi$.  Similarly, $P_{\nu_\phi}(\phi)-P_{\nu_s}(\phi)$ cannot be zero because $\nu_\phi\neq \nu_s$ for $s>0$.
\end{rmk}

 \begin{rmk}\label{rmk:st}
 \begin{itemize}
\item[(a)] Recall that $\overline\psi_n=\sum_{j=0}^{n-1}\overline\psi\circ T^j$
and that 
$\psi_T=\int_0^T\psi\circ \F_t\, dt$. It is known (see, for instance,~\cite[Theorem 2]{Sar06}) that in the setup of Theorem~\ref{thm:secmain} (a) with $\beta/\gamma<2$, $\frac{\overline\psi_n-n\int_{Y}\overline\psi\, d\mu_Y}{n^{\gamma/\beta}}\to^d M_{\beta/\gamma}$, where $M_{\beta/\gamma}$ is a random variable in the domain of a stable law with index
  $\beta/\gamma<2$. This lifts to a similar limit law for the flow (see, for instance,~\cite[Lemma 6.3]{BruTerTod21}): $\frac{\psi_T-T\int_{Y^\tau}\psi\, d\nu}{T^{\gamma/\beta}}\to^d M_{\beta/\gamma}$.
  
  In the setup of Theorem~\ref{thm:secmain} (a)  with $\beta/\gamma=2$, $\frac{\overline\psi_n-n\int_{Y}\overline\psi\, d\mu_Y}{\sqrt{n\log n}}\to^d \mathcal N(0,\sigma_0^2)$ for some non-zero $\sigma_0$ (see,~\cite[Theorem 3]{Sar06}).
  This is a Gaussian limit but with non-standard scaling $\sqrt{n\log n}$.
  The same type of limit lifts to the flow (see, for instance ~\cite[Lemma 6.3]{BruTerTod21}).
  
  In either of these two cases, that is $\beta/\gamma\in (1,2]$ in Theorem~\ref{thm:secmain} (a), the leading H{\"o}lder exponent depends on $\beta$ and $\gamma$.

  As soon as one has a CLT with standard normalization $\sqrt{n}$, as in  Theorem~\ref{thm:secmain}(b), the leading H{\"o}lder exponent is $1/2$, independent of $\beta$ and $\gamma$.
  Theorem~\ref{thm:secmain}  captures the transition from a stable law to the CLT with standard scaling in terms of the H{\"o}lder continuity of the pressure (in the weak* norm): the change in the H{\"o}lder exponent makes this precise.
  
 \item[(b)] We believe that some version of item Theorem~\ref{thm:secmain}(a) persists if one weakens the assumption to $\psi_0\in (C_1 \tau^\gamma, C_2\tau^\gamma)$ 
 with $C_1, C_2 >0$, 
  and even under weaker assumptions on the tail of $\tau$.
  In addition to the need to control the precise upper and lower bounds for $p'(s)-p'(0)$ in Proposition~\ref{prop:gmsec}(a)
  (which make the calculations seriously more cumbersome), one needs to ensure that $p''(s)>0$. This is very heavy in terms
  of calculations without assumptions that ensure regular variation of  $\psi_0$.  
   We do not pursue this here.
   \end{itemize}
  \end{rmk}

 \begin{rmk}\label{rmkkinkm}  We can interpret  \eqref{eq:heavyHold} and \eqref{eq:heavyquad} in Theorem~\ref{thm:secmain} (b)
 as follows: the pressure function has a
polynomial (in fact quadratic) form for $\beta/\gamma \in (2,3)$, but as $\beta/\gamma$ drops below
$2$, then the H\"older exponent jumps to $(\beta-\gamma+1)/(\beta-\gamma) > 1+1/\gamma > 2$.
This gives a kink in the second derivative of the pressure as function of the weak$^*$-norm of the measures.
This represents a phase transition of order $3$ if
$(\beta-\gamma+1)/(\beta-\gamma) \in (2,3)$ or of higher order if
$(\beta-\gamma+1)/(\beta-\gamma) \geq 3$.
\end{rmk}

\begin{rmk}
The EKP formula can fail to hold under our assumptions (GM0) and (GM1), when $\int\psi~d\nu<\int\psi~d\nu_\phi$.
We demonstrate this for
the Pomeau-Manneville map $f_\alpha:x \mapsto x(1+x^\alpha) \bmod 1$ on the unit interval
with $\alpha \in (0,1)$.
The induced map $T = f_\alpha^\tau$ on the domain $Y$ of the second branch is a full-branched Gibbs-Markov map.
 The potential $\phi = \log f'_\alpha$, so $\bar\phi = \log T'$, satisfies $P(\phi)=0$, and the equilibrium measure $\mu_{\bar\phi}$ is a Gibbs measure with $n^{-(\beta+1)}\ll \mu_{\bar\phi}(\tau = n) \ll n^{-(\beta+1)}$ for $\beta = 1/\alpha$.
Take the potential $\psi = C_0 \cdot 1_Y - C_1$ for some $C_0, C_1 > 0$, so
$\bar\psi(y) = C_0-\psi_0(y) = C_0-C_1\tau(y)$ where $C_0$ is sufficiently large that $\int\psi~d\nu_\phi>0$.

The partition $\{ a_k \}$ of $T$ has exactly one interval $a_k$ with $\tau|_{a_k} = k$ for each $k \geq 1$. Let $x_k \in a_k$ be such that $T(x_k) = x_k$, and let $\nu_k$ be the equidistribution on the orbit of $x_k$ under $f_\alpha$. The Gibbs property of $\mu_{\bar\phi}$, recalling that we assume $P(\phi)=0$, implies that
$e^{\bar\phi(x_k)} \gg \mu_{\bar\phi}(a_k) \gg k^{-(\beta+1)}$,
so $\bar\phi(x_k) \ge \log c - (\beta+1)\log k$ for some $c > 0$.

The lift of $\nu_k$ is the Dirac measure at $x_k$, so Abramov's formula gives $\int\phi~d\nu_k =  \frac{ \delta_{x_k}(\bar\phi)} { \delta_{x_k}(\tau)} \geq \frac{\log c - (\beta+1)\log k}{k}$.
Since also $h_{\nu_k}(f_\alpha) = 0$, we get
$$
0 = P(\phi) \geq P_{\nu_k}(\phi) = h_{\nu_k}(f_\alpha) + \int\phi~d\nu_k \ge
\frac{1}{k}\left( \log c - (\beta+1)\log k\right)\to 0.
$$
Finally notice that
\begin{align*}
 \int\psi~d\nu_k =\frac{\bar\psi(x_{a_k})}{k} = \frac{C_0-\psi_0(x_{a_k})}{k} =   \frac{C_0-C_1 k}{k}\to -C_1,
\end{align*}
as $k \to \infty$. Hence for any $C, \rho>0$ we can find $k$ such that
$$
\int\psi~d\nu_\phi-\int\psi~d\nu_k> C\left( P(\phi)-P_{\nu_k}(\phi)\right)^\rho,
$$
violating the EKP.

We stress that for other systems for which an induced map is Gibbs-Markov system with polynomial tail, we generally expect the same type of argument as above can be performed: the key, natural, requirement is that $\mu_{\bar\phi}(a_k) \gg k^{-(\beta+1)}$ for some infinite sequence of $k$.

We close this remark by pointing out that in this example the pressure function is not differentiable at 0.  Indeed for any $s < 0$, 
there is $k \in \N$ such that $p(s) = P(\phi + s\psi)  \geq \int \phi + s \psi~d\nu_k > -s C_1 - s^2$.
Therefore, the left derivative of $p(s)$ at zero is
$$
\lim_{s \nearrow 0} \frac{p(s)-p(0)}{s} \leq \lim_{s \nearrow 0} {-C_1-s} = -C_1 < 0.
$$
For $s \geq 0$, we have
$$
  p(s) = P(\phi + s\psi)  \geq \int \phi + s \psi~d\nu_\phi = P(\phi) + \int  s \psi~d\mu_\phi = s \int  \psi~d\nu_\phi,
$$
so the graph of the pressure function lies above a line with slope $\int  \psi~d\nu_\phi$.
Recall that we chose $C_0, C_1 > 0$ such that $\int  \psi~d\nu_\phi > 0$, so this slope is positive.
Since also the pressure function is convex, this implies that $p(s)$ is increasing for $s \geq 0$ and $p'(s) \geq \int \psi~ d\nu_\phi > 0$.
However, the left derivative of $p$ at $s = 0$ is negative so $p$ is not differentiable at $s=0$.
\label{rmk:EKPfail}
\end{rmk}

Finally we give an analogue of \cite[Theorem 7.1]{RuhSar22} in our setting, which handles the case when $\int\psi~d\nu$ is far from $\int\psi~d\nu_\phi$.  Note that our constant $C_{\phi, \psi}'$ is not very refined here, but also that we are dealing with some cases of unbounded potentials $\psi$, so we would not expect as much control as when we have boundedness.

\begin{thm}\label{thm:bigint} Assume (GM0) and (GM1). In the setup of Theorem~\ref{thm:firstmain} and \ref{thm:secmain} (b), let $\rho=1/2$. 
 In the setup of Theorem~\ref{thm:secmain} (a), let $\rho=\frac{\beta-\gamma}{\beta-\gamma+1}$.

 There exists $C_{\phi, \psi}'>0$ 
  so that for any $F$-invariant probability measure $\nu$ with 
  $\int\psi~d\nu> \int\psi~d\nu_\phi$, we have
 \[
\int\psi\,d\nu-\int\psi\,d\nu_{\phi} 
     \le C_{\phi, \psi}'\left(P_{\nu_\phi}(\phi)-P_\nu(\phi)\right)^\rho.
     \] 
\end{thm}

\section{Proof of Proposition~\ref{prop:gm}}
\label{sub:firstprop}

As is customary in the literature, due to the Ruelle-Perron-Frobenius (RPF) Theorem, in the setup of Gibbs-Markov maps $T:Y\to Y$ (see, for instance,~\cite[Section 3.3]{BruTerTod19}),  the study of
the pressure function $P_{T}(\overline{\phi+s\psi})$ comes down to the study of a perturbed version of the transfer operator $R:L^1(\mu_{\bar\phi})\to L^1(\mu_{\bar\phi})$. In particular, we identify $P_{T}(\overline{\phi+s\psi -u})$, $u\in [0,\delta)$, $s\in (0,\delta)$
for some $\delta>0$ with $\log\lambda(u,s)$, where $\lambda(u,s)$ is the leading eigenvalue
of the perturbed transfer operator 
\begin{equation*}
 R(u,s) v=R(e^{-u\tau} e^{s\overline{\psi}} v), \qquad u, s \in [0,\delta_0),\ v\in L^1(\mu_{\bar\phi}).
\end{equation*}
Note that by the argument at the end of Remark~\ref{rmk:p>0} coupled with Abramov's formula,  $\int\psi~d\nu_\phi>0$ implies that $P_T(\overline{\phi+s\psi})>0$ for $s>0$.
We briefly recall the application of the RPF Theorem.
Note that  $R(0,0)=R$ for $u=s=0$. We already know that $R$ has a spectral gap in $\cB_\theta$; in particular, this means that $1$ is a simple eigenvalue, isolated in the spectrum of $R$.
Under (GM1), there exists $\delta_0>0$ so that $\|R(u,s)-R(u,0)\|_{\cB_\theta}\ll s^\eps$, for some
$\eps>0$ and all $u, s \in [0,\delta_0)$. The proof of this fact is standard; for instance, it is an easier version of~\cite[Proof of Lemma 5.2]{BruTerTod19} ($\beta<1$ there gives some $\eps>0$ here). In fact, much more is true: see Lemma~\ref{lemma:sm} below.
Since we also know that $u\mapsto R(u,0)$ is analytic in  $u\in [0,\delta_0)$, there exists
a family of eigenvalues $\lambda(u,s)$ analytic in $u\in [0,\delta_0)$ and $C^1$ in $s\in [0,\delta_0)$
with $\lambda(0,0)=1$. By the RPF Theorem,
\begin{align}\label{eq:presindev}
 \bar p(u,s)=P_{T}(\overline{\phi+s\psi -u})=\log\lambda(u,s), \qquad  u, s\in [0,\delta_0).
\end{align}

To study the smoothness of $\lambda(u,s)$, as a function of $u$ and $s$, we need to recall some facts about the smoothness
of $R(u,s)$.

For non-integer $q_{*}\in\R_{+}$, we write $[q_{*}]$ for the integer part and say that a function $g:\R\to \R$ is $C^{q_*}$ 
if $|g|_{C^{[q_{*}]}}<\infty$ and $\sup_{x_1\neq x_2} |x_1-x_2|^{-(q_{*}-[q_{*}])}|\frac{\partial^{[q^*}}{\partial x^{[q_*]}}g(x_1)-\frac{\partial^{[q^*]}}{\partial x^{[q_*]}}g(x_2)|<\infty$. In a similar manner, we talk about the smoothness of $s\to R(u,s)$
and $u\to R(us,s)$. The statement of Lemma~\ref{lemma:sm} below makes this precise.

 Let $q_0$ and $q_1$ be as in~\eqref{tauq} and~\eqref{psi0q}.
Throughout, we write
\begin{align}\label{eq:gh}
 G_{[q_0]}(u,s)=\frac{\partial^{[q_0]}}{\partial u^{[q_0]}}R(u,s),\quad H_{[q_1]}(u,s)=\frac{\partial^{[q_1]}}{\partial s^{[q_1]}}R(u,s)
 \end{align}
and 
\begin{align}\label{eq:kk}
  K_{[q_1]}(u,s)=\frac{\partial^{[q_1]}}{\partial s^{[q_1]}}\frac{\partial}{\partial u} R(u,s).
\end{align}

\begin{lemma}\label{lemma:sm}
 Assume (GM0) and (GM1). Let $q_0$ and $q_1\in [1,\beta/\gamma)$ be so that~\eqref{tauq} and~\eqref{psi0q} hold.
 
 Let $G$, $H$ and $K$ as in~\eqref{eq:gh} and~\eqref{eq:kk}. Let $u, s \in [0,\delta_0)$. 
 Then $\|G_{[q_0]}(u,s)\|_{\cB_\theta}<\infty$ and $\|H_{[q_1]}(u,s)\|_{\cB_\theta}<\infty$. Moreover, there exists $C>0$ so that 
 \begin{itemize}\item[(i)]for all $u_1,u_2,s_1,s_2\in [0, \delta_0)$,
 \[
  \|G_{[q_0]}(u_1,s)-G_{[q_0]}(u_2,s)\|_{\cB_\theta}\le C|u_1-u_2|^{q_0-[q_0]},
  \]
    \[\|H_{[q_1]}(u,s_1)-H_{[q_1]}(u,s_2)\|_{\cB_\theta}\le C|s_1-s_2|^{q_1-[q_1]}.
 \]
  \item[(ii)] for all $u>0$ and $s_1,s_2 \in [0,\delta_0)$,
  \(
  \|K_{[q_1]}(u,s)\|_{\cB_\theta}\le C u^{\beta-q_1\gamma-1}
   \)
   and
  \[
  \|K_{[q_1]}(u,s_1)-K_{[q_1]}(u,s_2)\|_{\cB_\theta}\le C|s_1-s_2|^{q_1-[q_1]}\cdot u^{\beta-q_1\gamma-1}.
 \]
 \end{itemize}
\end{lemma}
\begin{rmk}\label{rmk:q1}
 Recall that under (GM1), $\gamma>\beta-1$. Hence, $q_1\in [1,\beta/\gamma)$ is so that
  $\beta-q_1\gamma<1$. This means that
 in Lemma~\ref{lemma:sm}(ii), the factor in $u$ blows up as $u\to 0$, but in a controlled way.
 \end{rmk}

\begin{proof}The first statements on $G_{[q_0]}(u,s)$ and $H_{[q_1]}$ follow immediately from~\cite[Proposition 12.1]{MelTer17}.
 Assumption (A1) there is part of (GM0), (GM1) here and the involved constants depend
 on the $L^{q_0}(\mu_{\overline{\phi}}), L^{q_1}(\mu_{\overline{\phi}})$ norm of $\tau,\bar\psi$ respectively, on $\theta\in (0,1)$
 and on the constants in (GM0), (GM1).
 
 We sketch the argument for the statement on $H_{[q_1]}$ and as a consequence, the somewhat easier fact that 
 $G_{[q_0]}(u,s)$ is $C^{q_1}$ in $s$.
 By the argument used in the proof of~\cite[Proposition 12.1]{MelTer17},
 for $w\in L^1(\mu_{\overline{\phi}})$ with $\essinf w>0$ and satisfying~\eqref{eq:inf},
we obtain
 \begin{align}\label{eq:use1}
 \|R(w\,v)\|_{\cB_\theta}\le C |w|_{L^1(\mu_{\overline{\phi}})}|v|_{\theta},
 \end{align}
for some $C>0$ depending on the constant appearing in \eqref{eq:inf}.

 Under (GM1), $\psi_0\in L^{q_1}(\mu_{\overline{\phi}})$. Since
 $H_{[q_1]}(u,s)\tilde v=R(\bar\psi^{[q_1]} e^{-u\tau} e^{s C_0}e^{-s\psi_0}\tilde v)$, the first statement
 on $H_{[q_1]}$ follows immediately from~\eqref{eq:use1} with $w=\bar\psi^{[q_1]}$ and $v=e^{-u\tau} e^{s C_0}e^{-s\psi_0}\tilde v$.
 Throughout the rest of the proof, we will heavily exploit~\eqref{eq:use1}, but we will not write down the explicit form of $w$ and $v$.\\[3mm]
 {\bf{Proof of item (i)}} Using~\eqref{eq:use1}, we compute that
 \begin{align*}
   \|(H_{[q_1]}(u,s_1)&-H_{[q_1]}(u,s_2))v\|_{\cB_\theta}
   \le \left\|R\left(\bar\psi^{[q_1]}(e^{s_1 C_0}-e^{s_2 C_0})e^{-s_1\psi_0}e^{-u\tau} v\right)\right\|_{\cB_\theta}\\
   &\quad +\left\|R(\bar\psi^{[q_1]}(e^{-s_1 \psi_0}-e^{-s_2 \psi_0})e^{s_2C_0}e^{-u\tau}v)\right\|_{\cB_\theta}\\
   &\le C_0|s_1-s_2|\left\|R(\bar\psi^{[q_1]} e^{-s\psi_0}v)\right\|_{\cB_\theta}
  + C\left|\bar\psi^{[q_1]}(e^{-s_1 \psi_0}-e^{-s_2 \psi_0})e^{-u\tau}\right|_{L^1(\mu_{\overline{\phi}})}|v|_{\theta}\\
  &\le C'|s_1-s_2|\,|\bar\psi^{[q_1]}|_{L^1(\mu_{\overline{\phi}})}|v|_{\theta}+ C\left|\bar\psi^{[q_1]}(e^{-s_1 \psi_0}-e^{-s_2 \psi_0})e^{-u\tau}\right|_{L^1(\mu_{\overline{\phi}})}|v|_{\theta},
     \end{align*}
 for some $C, C'>0$.
  
The second statement on $H_{[q_1]}$ follows since
\[
 |\bar\psi^{[q_1]}(e^{-s_1 \psi_0}-e^{-s_2 \psi_0}) e^{-u\tau}|_{L^1(\mu_{\overline{\phi}})}
\ll|s_1-s_2|^{q_1-[q_1]}\, |\psi_0^{q_1}|_{L^{1}(\mu_{\overline{\phi}})}\ll |s_1-s_2|^{q_1-[q_1]}.
\]

{\bf{Proof of item (ii)}} 
First, $K_{[q_1]}(u,0)=-R(\bar\psi^{[q_1]}\tau e^{-u\tau})$. Using~\eqref{eq:use1}, 
\(
\|(K_{[q_1]}(u,0)v\|_{\cB_\theta}\le C\left|\bar\psi^{[q_1]}\tau e^{-u\tau} \right|_{L^1(\mu_{\overline{\phi}})}.
\)
To estimate this quantity, let $S(x)=\mu_{\overline{\phi}}(\tau>x)$ and recall from Remark~\ref{rmk:q1} that $\beta-q_1\gamma<1$.  Integrating by parts and using (GM0),
\begin{align}\label{eq:utau}
 \nonumber \int_Y \tau^{q_1\gamma+1} & e^{-u\tau}\, d\mu_{\overline{\phi}}=- \int_0^{\infty} x^{q_1\gamma+1} e^{-ux}\, d(1-S(x))\\
 \nonumber &=(q_1\gamma+1)\int_0^\infty x^{q_1\gamma}(1-S(x)) e^{-ux}\, d x
 -u\int_0^\infty x^{q_1\gamma+1}(1-S(x)) e^{-ux}\, d x 
 \\
 \nonumber &\ll \int_0^\infty x^{-(\beta-q_1\gamma)} e^{-ux}\, d x+u\int_0^\infty x^{-(\beta-q_1\gamma+1)})e^{-ux}\, d x\\
 &\ll u^{\beta-q_1\gamma-1}
\left( \int_0^\infty t^{-(\beta-q_1\gamma)} e^{-t}\, d t+ \int_0^\infty t^{-\beta+q_1\gamma+1} e^{-t}\, d t\right)
 \nonumber \\
 &\ll u^{\beta-q_1\gamma-1}.
\end{align}
Hence, $\|(K_{[q_1]}(u,0)v\|_{\cB_\theta}\le Cu^{\beta-q_1\gamma-1}$, as claimed.

Using that $K_{[q_1]}(u,s)=-R(\bar\psi^{[q_1]}\tau e^{-u\tau} e^{s C_0}e^{-s\psi_0})$, we compute that  
\begin{align*}
 \left\|(K_{[q_1]}(u,s_1)-K_{[q_1]}(u,s_2))v\right\|_{\cB_\theta} &\le
 \left\|R\left(\bar\psi^{[q_1]}\tau(e^{s_1 C_0}-e^{s_2 C_0})e^{-s_1\psi_0}e^{-u\tau}\, v\right)\right\|_{\cB_\theta}\\
 &\quad +\left\|R\left(\bar\psi^{[q_1]}\tau(e^{-s_1 \psi_0}-e^{-s_2 \psi_0})e^{s_2C_0}e^{-u\tau}\, v\right)\right\|_{\cB_\theta}
\end{align*}
Using~\eqref{eq:use1} we obtain that there exists $C>0$ so that

\begin{align}\label{eq:estK}
 \nonumber\|K_{[q_1]}(u,s_1)-K_{[q_1]}(u,s_2)\|_{\cB_\theta}&\le C_0|s_1-s_2|\,  \left|\bar\psi^{[q_1]}\tau e^{-u\tau}\right|_{L^1(\mu_{\overline{\phi}})}\\
 &\qquad +C\left|\bar\psi^{[q_1]}\tau(e^{-s_1 \psi_0}-e^{-s_2 \psi_0})e^{-u\tau}\right|_{L^1(\mu_{\overline{\phi}})}.
\end{align}
Regarding the first term in \eqref{eq:estK}, recall (GM1) and note that $|\bar\psi^{[q_1]}\tau e^{-u\tau}|_{L^1(\mu_{\overline{\phi}})}\ll |\tau^{q_1\gamma+1} e^{-u\tau}|_{L^1(\mu_{\overline{\phi}})}$. This together with~\eqref{eq:utau} implies that the first term in~\eqref{eq:estK} is bounded by $|s_1-s_2|\, u^{\beta-q_1\gamma-1}$.
  
  It remains to estimate the second term in~\eqref{eq:estK}. Using (GM1), compute that
  \begin{align*}
  \left|\psi_0^{[q_1]}\tau(e^{-s_1 \psi_0}-e^{-s_2 \psi_0})e^{-u\tau}\right|_{L^1(\mu_{\overline{\phi}})}
  &\ll |s_1-s_2|^{q_1-[q_1]}\cdot|\psi_0^{q_1}\,\tau e^{-u\tau}|_{L^1(\mu_{\overline{\phi}})}\\
  &\ll |s_1-s_2|^{q_1-[q_1]}\cdot|\tau^{ q_1\gamma+1} e^{-u\tau}|_{L^1(\mu_{\overline{\phi}})}.
 \end{align*}
By~\eqref{eq:utau}, $|\tau^{ q_1\gamma+1} e^{-u\tau}|_{L^1(\mu_{\overline{\phi}})}\ll u^{\beta-q_1\gamma-1}$ and the conclusion follows.~\end{proof}

A consequence of Lemma~\ref{lemma:sm} is that the family of eigenvalues $\lambda(u,s)$ has 
`good' smoothness properties. Recall that $\tau^*,\bar\psi^*$ are as in Proposition~\ref{prop:gm} (ii).

\begin{cor}\label{cor:smev}The following hold in the setup of Lemma~\ref{lemma:sm}.  Let $u, s\in [0,\delta_0)$.
\begin{itemize}
 \item[(i)] $\lambda(u,s)=1+g(u,s)$, where $g(u,s)\to 0$ as $u,s\to 0$
 and $g(u,s)$ is $C^{q_0}$ in $u$ and $C^{q_1}$ in $s$.
 \item[(ii)] $\frac{\partial}{\partial u}\lambda(u,s)=-\tau^*+d(u,s)$, where $d(u,s)$ is $C^{q_0-1}$ in $u$ and $C^{q_1}$ in $s$
 and $d(u,0)\to 0$ as $u\to 0$.
 Moreover,
 $\frac{\partial}{\partial s}\lambda(u,s)=\bar\psi^*+e(u,s)$,
  where $e(u,s)$ is $C^{q_0}$ in $u$ and $C^{q_1-1}$ in $s$
  and $e(u,s)\to 0$ as $u,s\to 0$.
  
  \item[(iii)]Let
  $\kappa(u,s)=\frac{\partial}{\partial s}\frac{\partial}{\partial u}\lambda(u,s)$.
  Then for all $u,s\in [0,\delta_0)$,
  \(
  |\kappa(u,s)|\le C u^{\beta-q_1\gamma-1}   \)
   and $\kappa(u,s)$ is $C^{q_1-1}$ in $s$.
\end{itemize}
\end{cor}
\begin{proof}{\bf{(i)}.}
 Given that $v(u,s)$ is the normalized eigenvector corresponding to $\lambda(u,s)$,
 \begin{align}~\label{evnb}
 \nonumber 1 &-\lambda(u,s) =\int_Y (1-e^{-u\tau}e^{s\bar\psi})\, d\mu_{\overline{\phi}}
  -\int_Y (1-e^{-u\tau}e^{s\bar\psi})(v(0,0)-v(u,s))\, d\mu_{\overline{\phi}}\\
  \nonumber &:=\int_Y (1-e^{-u\tau}e^{s\bar\psi})\, d\mu_{\overline{\phi}} - V(u,s)\\
 &=\int_Y (1-e^{-u\tau})\, d\mu_{\overline{\phi}} - \int_Y (1-e^{s\bar\psi})\, d\mu_{\overline{\phi}} 
 +\int_Y (1-e^{-u\tau})(1-e^{s\bar\psi})\, d\mu_{\overline{\phi}} -V(u,s)\end{align}
 By Lemma~\ref{lemma:sm}, $V(u,s)\to 0$, as $u,s\to 0$ and item (i) follows.
 
 {\bf{(ii)}.} Using~\eqref{evnb}, compute that
 \begin{align*}
  -\frac{\partial}{\partial u}\lambda(u,s)&=\int_Y\tau\, d\mu_{\overline{\phi}}
  -\int_Y \tau(1-e^{-u\tau})\, d\mu_{\overline{\phi}}  -\int_Y\tau e^{-u\tau}(1-e^{s\bar\psi})\, d\mu_{\overline{\phi}}
  -\frac{\partial}{\partial u}V(u,s)\\
  &:=\int_Y\tau\, d\mu_{\overline{\phi}} +d(u,s).\end{align*}
 A calculation similar to the one used in obtaining~\eqref{eq:utau} (via (GM0) and (GM1)) shows that the functions
 $\int_Y \tau(1-e^{-u\tau})\, d\mu_{\overline{\phi}}$  and $\int_Y\tau e^{-u\tau}(1-e^{s\bar\psi})\, d\mu_{\overline{\phi}}$
 are $C^{q_0-1}$ in $u$  and also that $\int_Y\tau e^{-u\tau}(1-e^{s\bar\psi})\, d\mu_{\overline{\phi}}$
 is $C^{q_1}$ in $s$.
  Note that 
  \[
\frac{\partial}{\partial u}V(u,s)=\int_Y\tau e^{-u\tau}e^{s\bar\psi}(v(0,0)-v(u,s))\, d\mu_{\overline{\phi}}
-\int_Y (1-e^{-u\tau}e^{s\bar\psi})\, 
\frac{\partial}{\partial u}v(u,s))\, d\mu_{\overline{\phi}}.
 \]
  The required smoothness properties of $\frac{\partial}{\partial u}v(u,s)$ in $u$ and then in $s$,
  and as a consequence on $\frac{\partial}{\partial u}V(u,s)$, follow from
the statement on $G$ in Lemma~\ref{lemma:sm}(i) and from the statement on $K$ in Lemma~\ref{lemma:sm}(iii).
The statement on the smoothness of $\frac{\partial}{\partial u}\lambda(u,s)$ in $u$ and $s$ follows by putting all these together. 
Also, $d(u,0)=-\int_Y \tau(1-e^{-u\tau})\, d\mu_{\overline{\phi}} +O(u)$
and (by, for instance, the Dominated Convergence Theorem applied to $\int_Y \tau(1-e^{-u\tau})\, d\mu_{\overline{\phi}}$)
we obtain that $d(u,0)\to 0$ as $u\to 0$.

The statement on the smoothness of $\frac{\partial}{\partial s}\lambda(u,s)$ in $u$ and $s$ follows by a similar argument.

Item {\bf{(iii)}} is an immediate consequence of Lemma~\ref{lemma:sm}(ii).~\end{proof}

We can now proceed to

\begin{pfof}{Proposition~\ref{prop:gm}.}
 Throughout we will use Corollary~\ref{cor:smev} and the relation~\eqref{eq:presindev}.

{\bf{Proof of item (i)}.} 
Since $\bar p(u,s)=\log \lambda(u,s)$, using Corollary~\ref{cor:smev} (i) and (ii), 
 \begin{equation}\label{eq:expld1}
  \frac{\partial}{\partial u}\bar p(u,s)=\frac{\frac{\partial}{\partial u}\lambda(u,s)}{\lambda(u,s)}=-\tau^*+D(u,s),\quad \frac{\partial}{\partial s}\bar p(u,s)=\frac{\frac{\partial}{\partial s}\lambda(u,s)}{\lambda(u,s)} = \bar\psi^*+E(u,s),
 \end{equation}
 where
 \begin{itemize}
  \item[(a)] $D(u,s)$ is $C^{q_0-1}$ in $u$ and $C^{q_1}$ in $s$. Also, $D(u,0)\to 0$ as $u\to 0$.
  \item[(b)]$E(u,s)$ is $C^{q_0}$ in  $u$ and $C^{q_1-1}$ in $s$. Also, $E(u,0)\to 0$ as $u\to 0$.
 \end{itemize}
 
 In particular, $\bar p(0,s)=\lambda(0,s)-1+O(|1-\lambda(0,s)|^2)$
 and 
 \begin{align}\label{eq:derpos}
  \frac{\partial}{\partial s}\bar p(0,s)=\frac{\frac{\partial}{\partial s}\lambda(0,s)}{\lambda(0,s)}=\bar\psi^*+E(0,s),
 \end{align}
where $E(0,s)$ is $C^{q_1-1}$ in $s$.
 
 For use below in the proof of (ii), we also note that
 \begin{align}\label{eq:dounderpbar}
  \frac{\partial}{\partial s} D(u,s)&=\frac{\partial}{\partial s}\frac{\partial}{\partial u}\bar p(u,s)
  =  -\frac{\frac{\partial}{\partial u}\lambda(u,s)\frac{\partial}{\partial s}\lambda(u,s)}{\lambda(u,s)^2}+\frac{\frac{\partial}{\partial s}\frac{\partial}{\partial u}\lambda(u,s)}{\lambda(u,s)}\\
 \nonumber &=-\bar\psi^*\tau^*-E_0(u,s)+\frac{\frac{\partial}{\partial s}\frac{\partial}{\partial u}\lambda(u,s)}{\lambda(u,s)},
 \end{align}
where, using again Corollary~\ref{cor:smev} (i) and (ii),  $E_0(u,s)$ is $C^{q_0-1}$ in $u$ and $C^{q_1-1}$ in $s$. Moreover,  $\kappa(u,s)=\frac{\partial}{\partial s}\frac{\partial}{\partial u}\lambda(u,s)$ satisfies the properties listed in Corollary~\ref{cor:smev} (iii). In particular,  for all $u\in (0,\delta)$ and $s\in [0,\delta)$,
we have $|\kappa(u,s)|\ll u^{\beta-q_1\gamma-1}$ and $\kappa(u,s)$ is $C^{q_1-1}$ in $s$.
It follows that
\begin{align}\label{eq:expld2}
 \frac{\partial}{\partial s} D(u,s)=\frac{\partial}{\partial s}\frac{\partial}{\partial u}\bar p(u,s) =-\bar\psi^*\tau^*-E_1(u,s),
\end{align}
where $|E_1(u,s)|\ll u^{\beta-q_1\gamma-1}$
and  $E_1(u,s)$ is $C^{q_1-1}$ in $s$.\\

 {\bf{Proof of item (ii)}.} 
 We proceed via an `implicit equation'  exploited in~\cite[Proof of Theorem 4.1]{BruTerTod19}
 for the case $\beta<1$ (infinite equilibrium states). 
 The key new ingredient comes down to using the Implicit Function Theorem inside the above mentioned implicit equation.
 
 By (i), $r(u,s):=\frac{\partial}{\partial u}\bar p(u,s)$ is well defined.
 For any small $u_0>0$,
 \begin{align}\label{eq:impl}
  \bar p(u_0,s)-\bar p(0,s)=\int_0^{u_0} r(u,s)\, du=-\tau^* u_0+ \int_0^{u_0} D(u,s)\, du,
  \end{align}
where $D(u,s)$ is as in item (a) after~\eqref{eq:expld1}.

By liftability, for $u_0(s)=p(s)=P_{\F}(\phi+s\psi)$, we obtain $P_T(\overline{\phi+s\psi-u_0})=0$. 
Hence the LHS of~\eqref{eq:impl} is $-P_T(\overline{\phi+s\psi})$.
By assumption, $u_0(s)>0$ for all $s>0$. The continuity of the pressure function gives that
$u_0(s)\to 0$ as $s\to 0$. Thus,~\eqref{eq:impl}
holds, and
\begin{align}\label{eq:impl2}
-\bar p(0,s)=-\tau^* u_0(s)+ \int_0^{u_0(s)} D(u,s)\, du:=-\tau^* u_0(s) +L(u_0(s),s).
  \end{align}
At this point we can conclude that, as $s\to 0$,
\begin{align}\label{eq:p}
p(s)=u_0(s)=\frac{\bar p(0,s)}{\tau^*}+\frac{L(u_0(s),s)}{\tau^*}=\frac{\bar p(0,s)}{\tau^*}(1+o(1))=s \frac{\bar\psi^*}{\tau^*}(1+o(1)).
\end{align}
 The first equality is by definition. The second equality 
follows immediately from~\eqref{eq:impl2}, while in the third we used the smoothness of $D(u,s)$ in $s$
and the fact that $D(u,0)\to 0$ (as in item (a) after~\eqref{eq:expld1}). 
The fourth equality follows  from~\eqref{eq:derpos}, since $E(0,s)$ is $C^{q_1-1}$ in $s$.
 
 We continue with the study of the derivative in $s$ of $u_0(s)$ via~\eqref{eq:impl2}. From here on we write $u_0:=u_0(s)$.
 
 Since $D(u,s)$ is uniformly continuous in $u$ (since it is $C^{q_0-1}$ in $u$), 
$\frac{\partial}{\partial u_0} L(u_0,s)=D(u_0,s)$, for all $s$. Set
\[
 M(u_0,s):=L(u_0,s)+\bar p(0,s),
\]
and note that $\frac{\partial}{\partial u_0} M(u_0,s)=D(u_0,s)\ne 0$, for all $u_0, s$ small enough. 
Since $M(u_0,s) - \tau^* u_0(s) \equiv 0$ and we also know that $|\frac{\partial}{\partial u_0} L(u_0,s)|<\infty$ and $|\frac{\partial}{\partial s} L(u_0,s)|<\infty$ (because $D(u_0,s)$ is $C^q_1$ in $s$), the IFT ensures that
$u_0(s)$ is differentiable in $s$ and
\begin{align}\label{eq:derivp}
 u_0'(s)= \frac{\frac{\partial}{\partial s} M(u_0,s)}{\tau^* - \frac{\partial}{\partial u_0} M(u_0,s)}.
\end{align}
We first estimate the numerator in~\eqref{eq:derivp}. Using~\eqref{eq:derpos},
\begin{align*}
 \frac{\partial}{\partial s} M(u_0,s)=\frac{\partial}{\partial s} L(u_0,s)+\bar\psi^*+E(0,s),
\end{align*}
where $E(0,s)$ is $C^{q_1-1}$ in $s$. Using the definition of $L(u_0,s)$ in~\eqref{eq:impl2} and also recalling~\eqref{eq:expld2},
\begin{align}\label{eq:dersl}
 \left|\frac{\partial}{\partial s} L(u_0,s)\right|&=\left|\int_0^{u_0}\frac{\partial}{\partial s}D(u,s)\, du\right|
 =\left|\int_0^{u_0}\frac{\partial}{\partial s}
 \frac{\partial}{\partial u}\bar p(u,s)\, du\right|\\
 \nonumber&=\left|\bar\psi^*\tau^*u_0+\int_0^{u_0}E_1(u,s)\, du\right|\ll u_0 +u_0^{\beta-q_1\gamma}.
\end{align}
Moreover, using the smoothness properties of $E_1$, we obtain that $\frac{\partial}{\partial s} L(u_0,s)$ is $C^{q_1-1}$ in $s$.
Thus,
\begin{align}\label{eq:m1} \frac{\partial}{\partial s} M(u_0,s) =\bar\psi^*+\hat E(u_0,s),
 \end{align}
 where $\hat E$  is well-defined in $u_0$  and $C^{q_1-1}$
 in $s$.

We continue with estimating the denominator in~\eqref{eq:derivp}. 
Recall that $\frac{\partial}{\partial u_0} M(u_0,s)=D(u_0,s)$, where $D$ is as in item (a) after~\eqref{eq:expld1}. In particular, $D(u_0,s)$ is $C^{q_1}$ in $s$. 
By~\eqref{eq:p}, $u_0(s)=O(s)$. Using the smoothness of $D(u_0,s)$ in $s$, we note that 
\[
\frac{1}{\tau^*-\frac{\partial}{\partial u_0} M(u_0,s)}=\frac{1}{\tau^* - D(u_0,s)}=\frac{1}{\tau^*}\left(1+O(D(u_0,s))\right)^{-1}
= \frac{1+o(1)}{\tau^*} \quad \text{ as } s\to 0.
\]
Recalling the smoothness properties of  $\hat E(u_0,s)$ in~\eqref{eq:m1}, we obtain $p'(0)=\frac{\bar\psi^*}{\tau^*}$.\\

{\bf{Proof of item (iii)}.} 
When $q_1>2$, differentiating in~\eqref{eq:p},
\begin{align*}
 p''(s)=\frac{\frac{\partial^2}{\partial s^2}\bar p(0,s)}{\tau^*}+\frac{\frac{\partial^2}{\partial s^2}L(u_0,s)}{\tau^*}.
\end{align*}

A very lengthy but straightforward\footnote{A refined version of this calculation is covered inside the proof of Proposition~\ref{prop:gmsec}. See in particular,~\eqref{eq:reftoit}, which deals with the case $q_1=\beta/\gamma<2$. The calculations are the same, just the exponent is different: see Remark~\ref{rmk:reftoit}.} calculation based on the smoothness properties of the function $D(u_0,s)$ 
(after differentiating~\eqref{eq:dersl} once more in $s$) shows that
$\frac{\partial^2}{\partial s^2}L(u_0,s)=o(1)$ as $s\to 0$.

Finally, it is known (see~\cite[Theorem 3]{Sar06})
that $\frac{\partial^2}{\partial s^2}\bar p(0,s)\Big|_{s=0}=\bar\sigma^2$, with $\bar\sigma^2$ as defined in~\eqref{eq:sigma}. Thus, $p''(0)=\frac{\bar\sigma^2}{\tau^*}$,
and the conclusion follows from the first equality in~\eqref{eq:sigma}.~\end{pfof}

\section{Refined estimates in the setup of Proposition~\ref{prop:gmsec}}
\label{sec:refined}

We start with a refined version of Lemma~\ref{lemma:sm}. Recall from~\eqref{eq:gh} and~\eqref{eq:kk}
 that $H_{[q_1]}(u,s)v=\frac{\partial}{\partial s^{[q_1]}}R(u,s)v= R(\bar\psi^{[q_1]} e^{-u\tau} e^{s\bar\psi}v)$ and that
$K_{[q_1]}(u,s)v=\frac{\partial}{\partial s^{[q_1]}}\frac{\partial}{\partial u} R(u,s)v=-R(\bar\psi^{[q_1]} \tau e^{-u\tau} e^{s\bar\psi}v)$. 
In Lemma~\ref{lemma:sm} we dealt with the continuity properties of $H$ and $K$ as $u,s\to 0$.
The first result below tells us how the derivatives in $s$ of $H$ and $K$ go to $\infty$ as $u,s\to 0$.

\begin{lemma}\label{lemma:ref} Assume the setup of Proposition~\ref{prop:gmsec}, in particular $\gamma \in (\beta-1,\beta)$.
Let $u, s\in [0, \delta_0)$. 
\begin{itemize}
\item[(i)] If $[q_1]=1$ and $\beta/\gamma\in (1,2]$ then
 $\|H_{1}(u,s)\|_{\cB_\theta} < \infty$ and
  \(
  \|K_{1}(u,0)\|_{\cB_\theta}\le C u^{\beta-\gamma-1},
   \)
   for some $C>0$. 
   
   Furthermore, if $\beta/\gamma\in (1,2)$, there exist $ C_2, C_3, C_4>0$ so that
  \[
  \left\|\frac{\partial}{\partial s}H_{1}(u,s)\right\|_{\cB_\theta}\le C_2 u^{\beta-2\gamma},\quad  \left\|\frac{\partial}{\partial s}K_{1}(u,s)\right\|_{\cB_\theta}\le C_3 u^{\beta-2\gamma-1}.
   \]
   and
   \[
  \left\|\frac{\partial}{\partial s}H_{1}(0,s)\right\|_{\cB_\theta}\le C_4 s^{\beta/\gamma-2}.
   \]   
   
   If $\beta/\gamma=2$, then there exist $C_2, C_3, C_4>0$ so that
   
   \[
  \left\|\frac{\partial}{\partial s}H_{1}(u,s)\right\|_{\cB_\theta}\le C_2 \log(1/u),\quad  \left\|\frac{\partial}{\partial s}K_{1}(u,s)\right\|_{\cB_\theta}\le C_3 u^{-1}.
   \]
   and
   \[
  \left\|\frac{\partial}{\partial s}H_{1}(0,s)\right\|_{\cB_\theta}\le C_4 \log(1/s).
   \]     

\item[(ii)] If $[q_1]=2$ and $\beta/\gamma\in (2,3)$ then
 $\|H_{2}(u,s)\|_{\cB_\theta} < \infty$ and
  \(
  \|K_{2}(u,0)\|_{\cB_\theta}\le C u^{\beta-2\gamma-1},
   \)
   for some $C>0$. Furthermore, there exist $C_2, C_3, C_4>0$ so that
  \[
 \left \|\frac{\partial}{\partial s}H_{2}(u,s)\right\|_{\cB_\theta}\le C_2 u^{\beta-3\gamma},\quad  \left\|\frac{\partial}{\partial s}K_{2}(u,s)\right\|_{\cB_\theta}\le C_3 u^{\beta-3\gamma-1}.
   \] 
   and
   \[
  \left\|\frac{\partial}{\partial s}H_{2}(0,s)\right\|_{\cB_\theta}\le C_4 s^{\beta/\gamma-3}.
   \]    \end{itemize}
\end{lemma}

\subsection{Some general type of integrals}
\label{sec:integr}

Before proving Lemma~\ref{lemma:ref}, we provide estimates of some general types of integrals. These or variants of them will be used throughout the proofs of the technical results in this section. Let $S(x)=\mu_{\overline{\phi}}(\tau<x)$ and recall from (GM1) that $\gamma>\beta-1$, so $\beta-\gamma<1$.
Since $1-S(x)=c x^{-\beta}(1+o(1))$,
\begin{align}\label{eq:utauprec}
 \nonumber \int_Y \tau^{\gamma+1} e^{-u\tau}\, d\mu_{\overline{\phi}} &=- \int_0^{\infty} x^{\gamma+1} e^{-ux}\, d(1-S(x))\\
 \nonumber &=(\gamma+1)\int_0^\infty x^{\gamma}(1-S(x)) e^{-ux}\, dx - u\int_0^\infty x^{\gamma+1}(1-S(x)) e^{-ux}\, d x 
 \\
 \nonumber &=c(\gamma+1)(1+o(1))\int_0^\infty e^{-ux} x^{-\beta+\gamma}\, dx - u(1+o(1)\, c\int_0^\infty e^{-ux} x^{\gamma+1-\beta}\, d x\\
 \nonumber &= c u^{\beta-\gamma-1}(1+o(1))
 \left((\gamma+1) \int_0^\infty e^{-t}t^{-\beta+\gamma} \, d t - \int_0^\infty e^{-t}t^{-\beta+\gamma+1}, d t\right)\\
 & =C u^{\beta-\gamma-1}(1+o(1)),
 \end{align}
 for a positive $C$ depending only on $c,\beta,\gamma$.
 
 By a similar argument, if $\beta/\gamma\ne 2$, then
 \begin{equation}\label{eq:utauneq2}
\begin{cases}  \int_Y \tau^{2\gamma} e^{-u\tau}\, d\mu_{\overline{\phi}}= Cu^{\beta-2\gamma}(1+o(1)),\\[2mm]
 \int_Y \tau^{2\gamma+1} e^{-u\tau}\, d\mu_{\overline{\phi}}= C'u^{\beta-2\gamma-1}(1+o(1))
 \end{cases}
 \end{equation}
for some $C, C' > 0$, whereas if $\beta/\gamma=2$ then
\begin{equation}\label{eq:utaueq2}
 \begin{cases}
 \int_Y \tau^{2\gamma} e^{-u\tau}\, d\mu_{\overline{\phi}}=  \int_Y \tau^{\beta} e^{-u\tau}\, d\mu_{\overline{\phi}}=C\log(1/u)(1+o(1)),\\[2mm]
 \int_Y \tau^{2\gamma+1} e^{-u\tau}\, d\mu_{\overline{\phi}}=  \int_Y \tau^{\beta+1} e^{-u\tau}\, d\mu_{\overline{\phi}}=Cu^{-1}(1+o(1)).
 \end{cases}
\end{equation}
 Recall that $\bar\psi=C_0-\psi_0=C_0-C_1\tau^\gamma$.
 Similar calculations, this time with $S(x)=\mu_{\overline{\phi}}(\psi_0<x)=\mu_{\overline{\phi}}(C_1\tau^\gamma<x)$, show that
 if $\beta/\gamma<2$, $\int_Y \psi_0^{2} e^{-s\psi_0}\, d\mu_{\overline{\phi}} =Cs^{\beta/\gamma-2}(1+o(1))$ for some $C>0$
 and that if $\beta/\gamma\in (2,3)$, $\int_Y \psi_0^{3} e^{-s\psi_0}\, d\mu_{\overline{\phi}} =-Cs^{\beta/\gamma-3}(1+o(1))$
 for some $C>0$. The involved constant depend only on $c, \beta, \gamma$.
 If $\beta/\gamma=2$ then
 \(
 \int_Y \psi_0^{2} e^{-s\psi_0}\, d\mu_{\overline{\phi}} =C\log(1/s)(1+o(1)).
  \)
  The involved constants (denoted by $C$ here) depend only on $c, \beta, \gamma$.

 Next, note that $\bar\psi^2=C_0^2+\psi_0^2-2C_0\psi_0$ and that $\bar\psi^3=C_0^3-\psi_0^3+3C_0^2\psi_0-3C_0\psi_0^2$. Thus,
 there exist $C_2, C_3, C_4$ depending only on $c, \beta, \gamma$ so that
  \begin{equation}\label{eq:utaus}
\begin{cases}
 \int_Y \bar\psi^{2} e^{s\bar\psi}\, d\mu_{\overline{\phi}} =C_2s^{\beta/\gamma-2}(1+o(1)), & \text{ if } \beta/\gamma<2,\\[2mm]
 \int_Y \bar\psi^{2} e^{s\bar\psi}\, d\mu_{\overline{\phi}} =C_3\log(1/s)(1+o(1)), & \text{ if } \beta/\gamma=2,\\[2mm]
\int_Y \bar\psi^{3} e^{s\bar\psi}\, d\mu_{\overline{\phi}} =-C_4s^{\beta/\gamma-3}(1+o(1)) & \text{ if } \beta/\gamma \in (2,3).\\
\end{cases}
\end{equation}

 \begin{pfof}{Lemma~\ref{lemma:ref}} We provide the argument for item (i). Item (ii) follows by a similar argument after taking one more
derivative in $s$.

The first estimate on $H_1$ follows directly from Lemma~\ref{lemma:sm} with $[q_1]=1$.

Next, note that if $\beta/\gamma\in (1,2)$,
\[
\left\|\frac{\partial}{\partial s}H_{1}(u,s)\right\|_{\cB_\theta}\ll \| R(\bar\psi^{2} e^{-u\tau})\|_{\cB_\theta}\ll |\tau^{2\gamma}e^{-u\tau}|_{L^{1}(\mu_{\overline{\phi}})}\ll u^{\beta-2\gamma},
\]
where we used the first equation in~\eqref{eq:utauneq2}. The estimate for the case $\beta/\gamma=2$ follows similarly using~\eqref{eq:utaueq2}.
Also,  if $\beta/\gamma\in (1,2)$,
\[
\left\|\frac{\partial}{\partial s}H_{1}(0,s)\right\|_{\cB_\theta}\ll \| R(\bar\psi^{2} e^{s\bar\psi})\|_{\cB_\theta}\ll |\bar\psi^{2} e^{s\bar\psi}|_{L^{1}(\mu_{\overline{\phi}})}\ll 
s^{\beta/\gamma-2},
\]
where we have used the first estimate of~\eqref{eq:utaus} for $s$. The estimate for the case $\beta/\gamma=2$ follows similarly using the corresponding estimate of~\eqref{eq:utaus} for this case.

Regarding $K_1$,  if $\beta/\gamma\in (1,2)$,
\[
\left\|\frac{\partial}{\partial s}K_{1}(u,s)\right\|_{\cB_\theta}\ll \| R(\bar\psi^{2} \tau e^{-u\tau})\|_{\cB_\theta}\ll |\tau^{2\gamma+1}e^{-u\tau}|_{L^{1}(\mu_{\overline{\phi}})}\ll u^{\beta-2\gamma-1},
\]
where we used the second equation in~\eqref{eq:utauneq2}. The estimate for the case $\beta/\gamma=2$ follows similarly using the corresponding estimates for this case.~\end{pfof}

We shall also need the following refined version of Corollary~\ref{cor:smev} (ii) and (iii).
Item (i) of Corollary~\ref{cor:smev} remains unchanged.
Again, the derivatives in $s$ of several quantities in the lemma below blow up as $u,s\to 0$ but in a controlled way.

We recall that in the setup of Proposition~\ref{prop:gmsec}, $\gamma<1$ and $\beta<2$.

\begin{lemma}\label{lemm:refev}
The following hold in the setup of Proposition~\ref{prop:gmsec}.  Let $u, s\in [0,\delta_0)$.
\begin{itemize}
 \item[(i)] $\frac{\partial}{\partial u}\lambda(u,s)=-\tau^*+d(u,s)$, where $d(u,s)$ is as follows.
 
 There exists $C>0$ depending only on $c,\beta$ so that $d(u,0)=
 Cu^{\beta-1}(1+o(1))$.  
 Moreover, there exist $C_2, C_3>0$ depending only on $c,\beta, \gamma$ so that
 as $u,s\to 0$,
 \begin{align*}
 \frac{\partial}{\partial s} d(u,s)=
 C_2 u^{\beta-\gamma-1}(1+o(1)), &\text{ if }\beta/\gamma\in (1,2],
   \end{align*}
 
  \begin{align*}
 \frac{\partial^2}{\partial s^2} d(u,s)=
 C_3 u^{\beta-\gamma-2}(1+o(1)), &\text{ if }\beta/\gamma\in (2,3).
  \end{align*}

 \item[(ii)] The following holds for some $C, C'>0$ depending only on $c,\beta/\gamma$.
  \begin{align*}
\frac{\partial}{\partial s}\lambda(u,s)=\bar\psi^*+ e(u,s) + \begin{cases}
h(s)+ h_0(s), &\text{ if }\beta/\gamma\in (1,2],\\
 -s\int_Y \bar\psi^2\, d\mu_{\overline{\phi}}+ Cs^{\beta/\gamma-1}+h_1(s), &\text{ if }\beta/\gamma\in (2,3),
 \end{cases}
  \end{align*} 
  where $h(s)=Cs^{\beta/\gamma-1}$ if  $\beta/\gamma\in (1,2)$, $h(s)=C\log(1/s)$ if  $\beta/\gamma=2$
  and where
    $h_0$, $h_1$  and $e$ are as follows: 
 \begin{itemize}
 \item[(a)] $h_0(s)=o(s^{\beta/\gamma-1})$, $h'_0(s)=o(s^{\beta/\gamma-2})$ if $\beta/\gamma\in (1,2)$
 and $h'_0(s)=o(\log(1/s))$ if $\beta/\gamma=2$.
 
 \item[(b)] $h_1(s)=o(s^{\beta/\gamma-1})$, $h'_1(s)=C' s^{\beta/\gamma-2}(1+o(1))$
  and $h''_1(s)=C' s^{\beta/\gamma-3}(1+o(1))$;
  
 \item[(c)] $e(0,s)=O(s)$, $e(u,0)=o(1)$ as $u,s\to 0$ and

  \begin{itemize}
  \item[(*)] If $\beta/\gamma\in (1,2)$, then $\frac{\partial}{\partial s} e(u,s)= o(u^{\beta-\gamma-1})+o(s^{\beta/\gamma-2})$.
  Also, $\frac{\partial}{\partial s} e(0,s)=o(s^{\beta/\gamma-2})$.
   
   \item[(**)] If $\beta/\gamma=2$, then $\frac{\partial}{\partial s} e(u,s)= o(u^{\beta-\gamma-1})+o(\log(1/s))$. Also, $\frac{\partial}{\partial s} e(0,s)=o(\log(1/s))$.
    
   \item[(***)] If $\beta/\gamma\in (2,3)$, then $\frac{\partial}{\partial s} e(u,s)= o(u^{\beta-\gamma-2})+o(s^{\beta/\gamma-3})$.   Also, $\frac{\partial}{\partial s} e(0,s)= o(s^{\beta/\gamma-3})$.    \end{itemize}
    \end{itemize}

  \item[(iii)] Let
  $\kappa(u,s)=\frac{\partial}{\partial s}\frac{\partial}{\partial u}\lambda(u,s)$.
  Then there exist $C, C'>0$ depending only on $c,\beta,\gamma$,
  so that
  $$
  \begin{cases}
  \kappa(u,0)=C u^{\beta-\gamma-1} + O(u^{\beta-\gamma-1+\eps_0}), & \text{ if } \beta/\gamma\in (1,2], \\
  \frac{\partial}{\partial s}\kappa(u,s)\Big|_{s=0}=C' u^{\beta-2\gamma-1} + O(u^{\beta-2\gamma-1+\eps_0}), & \text{  if } \beta/\gamma\in (2,3),
  \end{cases}
  $$
   as $u\to 0$ and for any $\eps_0>0$.

   Also, the following hold for some $\hat C_2, \hat C_3>0$ depending only on $c,\beta,\gamma$, as $u,s\to 0$.
   \begin{itemize}
    \item[(*)] If $\beta/\gamma\in (1,2]$, then $\frac{\partial}{\partial s}\kappa(u,s)=\hat C_2 u^{\beta-2\gamma-1}(1+o(1))$.
    \item[(**)] If $\beta/\gamma\in (2,3)$, then $\frac{\partial^2}{\partial s^2}\kappa(u,s)=-\hat C_3 u^{\beta-3\gamma-1}(1+o(1))$.    \end{itemize}
    \end{itemize}
\end{lemma}

\begin{pfof}{Lemma~\ref{lemm:refev}}
We continue from the proof of Corollary~\ref{cor:smev} (ii) with the same notation.\\[2mm]
\textbf{Proof of item (i)}
Recall that 
\begin{align}\label{derivu}
\nonumber  -\frac{\partial}{\partial u}\lambda(u,s)&=\int_Y\tau\, d\mu_{\overline{\phi}}
  -\int_Y \tau(1-e^{-u\tau})\, d\mu_{\overline{\phi}}  -\int_Y\tau e^{-u\tau}(1-e^{s\bar\psi})\, d\mu_{\overline{\phi}}
 \\
 \nonumber &-\int_Y\tau e^{-u\tau}e^{s\bar\psi}(v(0,0)-v(u,s))\, d\mu_{\overline{\phi}}+\int_Y (1-e^{-u\tau}e^{s\bar\psi})\, \frac{\partial}{\partial u}v(u,s)\, d\mu_{\overline{\phi}}\\
 &:=\int_Y\tau\, d\mu_{\overline{\phi}}
  -\int_Y \tau(1-e^{-u\tau})\, d\mu_{\overline{\phi}} -W_0(u,s) -W_1(u,s)-W_2(u,s).\end{align}
 Recall  $\mu_{\overline\phi}(\tau\ge x)=c x^{-\beta}(1+o(1))$. A standard calculation (mostly similar to the one used  in obtaining~\eqref{eq:utauprec}) shows that
 there exists $C>0$ depending on $c$ and $\beta$ so that
 \begin{align*}
- \int_Y \tau(1-e^{-u\tau})\, d\mu_{\overline{\phi}}=
 Cu^{\beta-1}(1+o(1)).
  \end{align*}
    Set $d(u,s)= \int_Y \tau(1-e^{-u\tau})\, d\mu_{\overline{\phi}} -W_0(u,s) -W_1(u,s)-W_2(u,s)$ with $W_0, W_1, W_2$ as defined in~\eqref{derivu}.
  Note that $W_0(u,0)=0$, $|W_1(u,0)|\ll u$ and $|W_2(u,0)|\ll u$ and that so far we obtained the expression for $d(u,0)$.
  
  Note that $\frac{\partial}{\partial s}d(u,s)=-\frac{\partial}{\partial s}  (W_0(u,s) +W_1(u,s)+W_2(u,s))$. 
  We continue with the derivatives in $s$ of $W_0, W_1, W_2$ by considering each of the two cases. \\

  \textbf{The term  $\bm{W_0(u,s)}$.}
  First, $\frac{\partial}{\partial s}W_0(u,s) =\int_Y\tau \bar\psi e^{-u\tau}e^{s\bar\psi}\, d\mu_{\overline{\phi}}
  =\int_Y\tau \bar\psi e^{-u\tau}\, d\mu_{\overline{\phi}}+ \int_Y\tau \bar\psi e^{-u\tau} (e^{s\bar\psi}-1)\, d\mu_{\overline{\phi}}$.
  
  \textbf{ If $\bm{\beta/\gamma\in (1,2]}$},  then $\beta-\gamma\in (0,1)$. 
  Since $\bar\psi=C_0-C_1\tau^\gamma$,
 \begin{align}\label{us11}
\int_Y\tau \bar\psi e^{-u\tau}\, d\mu_{\overline{\phi}}&=C_0\int_Y\tau e^{-u\tau}\, d\mu_{\overline{\phi}} -C_1\int_Y \tau^{\gamma+1} e^{-u\tau}\, d\mu_{\overline{\phi}} \nonumber \\
&=-C u^{\beta-\gamma-1}(1+o(1)),
\end{align} 
for some $C>0$ depending on $c$ and $\beta, \gamma$. In the last equality we have used that~\eqref{eq:utauprec}  holds as soon as $\beta-\gamma\in (0,1)$.
Since we also know that  that $e^{s\bar\psi}-1\to 0$ as $s\to 0$,  the Dominated Convergence Theorem implies that $\int_Y\tau \bar\psi e^{-u\tau} (e^{s\bar\psi}-1)\, d\mu_{\overline{\phi}}=o(u^{\beta-\gamma-1})$.
So, if $\beta/\gamma\in (1,2]$ then  $\frac{\partial}{\partial s}W_0(u,s) =-C u^{\beta-\gamma-1}(1+o(1))$.

  \textbf{ If $\bm{\beta/\gamma\in (2,3)}$},  then $\beta-2\gamma<\gamma<1$ and $\beta-2\gamma\in (0,\gamma)\subset(0,1)$.
  Note that $\frac{\partial^2}{\partial s^2}W_0(u,s) =\int_Y\tau \bar\psi^2 e^{-u\tau}e^{s\bar\psi}\, d\mu_{\overline{\phi}}$.
Proceeding similarly to the argument above in the case  $\beta/\gamma\in (1,2]$, we compute that if $\beta-2\gamma\in (0,1)$,
then 
$\frac{\partial^2}{\partial s^2}W_0(u,s)=C u^{\beta-2\gamma-1}(1+o(1))$
for some $C$ depending on $c$ and $\beta, \gamma$, where we use an analogue of~\eqref{eq:utauprec} for the case $\beta-2\gamma\in (0,1)$. 
So, if $\beta/\gamma\in (2,3)$ then  $\frac{\partial}{\partial s}W_0(u,s) =-C u^{\beta-2\gamma-1}(1+o(1))$.\\

\textbf{The term  $\bm{W_1(u,s)}$.}
Start from
\begin{align*}
\frac{\partial}{\partial s}W_1(u,s)=\int_Y\tau\bar\psi e^{-u\tau}e^{s\bar\psi}(v(u,0)-v(u,s))\, d\mu_{\overline{\phi}}-\int_Y\tau e^{-u\tau}e^{s\bar\psi}\, \frac{\partial}{\partial s}v(u,s)\, d\mu_{\overline{\phi}}.
\end{align*}

 Recall that \textbf{if $\bm{\beta/\gamma\in (1,2]}$},  then 
 $\beta-\gamma\in (0,1)$. Since
 \[
  \|v(0,0)-v(u,s)\|_{\cB_\theta}\le  \|v(0,s)-v(u,s)\|_{\cB_\theta} +\|v(u,0)-v(u,s)\|_{\cB_\theta}\le u+s,
\]
 using~\eqref{us11}, we obtain $\int_Y\tau\bar\psi e^{-u\tau}e^{s\bar\psi}(v(u,0)-v(u,s))\, d\mu_{\overline{\phi}}=o(u^{\beta-\gamma-1})$, as $u,s\to 0$.
 Also, by Lemma~\ref{lemma:ref}(i) (statement on $H_1$), $\|\frac{\partial}{\partial s}v(u,s)\|_{\cB_\theta}<\infty$.
 Recall $e^{s\bar\psi}\ll e^{sC_0}e^{-s\tau^\gamma}$. Thus,
 \[
  \left|\int_Y\tau e^{-u\tau}e^{s\bar\psi}\, \frac{\partial}{\partial s}v(u,s))\, d\mu_{\overline{\phi}}\right|\ll \int_Y\tau e^{-u\tau}e^{s\bar\psi}\, d\mu_{\overline{\phi}}\ll \int_Y\tau\, d\mu_{\overline{\phi}} =O(1).
  \]
 Thus, if $\beta/\gamma\in (1,2]$, $\frac{\partial}{\partial s}W_1(u,s)= o(u^{\beta-\gamma-1})$, as $u,s\to 0$.

  Next, recall that \textbf{if $\bm{\beta/\gamma\in (2,3)}$},  then 
 $\beta-2\gamma\in (0,1)$.  In this case, taking one more derivative, 
 \begin{align*}
 \frac{\partial^2}{\partial s^2}W_1(u,s)& =\int_Y\tau\bar\psi^2 e^{-u\tau}e^{s\bar\psi}(v(u,0)-v(u,s))\, d\mu_{\overline{\phi}}
-\int_Y\tau\bar\psi e^{-u\tau}e^{s\bar\psi}\, \frac{\partial}{\partial s}v(u,s)\, d\mu_{\overline{\phi}}\\
&\quad-\int_Y\tau e^{-u\tau}e^{s\bar\psi}\, \frac{\partial^2}{\partial s^2}v(u,s)\, d\mu_{\overline{\phi}}=:I_1+I_2+I_3\end{align*}

Using the analogue of~\eqref{us11}  for the case $\beta-2\gamma\in (0,1)$,
\[
|I_1|\ll \|v(0,0)-v(u,s)\|_{\cB_\theta} \int_Y\tau\bar\psi^2 e^{-u\tau}\, d\mu_{\overline{\phi}} \ll u^{\beta-2\gamma-1} (u+s)=o(u^{\beta-2\gamma-1})
\]
 as $u,s\to 0$.

Next, we already know that  $\|\frac{\partial}{\partial s}v(u,s)\|_{\cB_\theta}<\infty$. Thus, $|I_2|\ll \int_Y\tau^{\gamma+1} e^{-u\tau}e^{s\bar\psi}\, d\mu_{\overline{\phi}}\ll u^{\beta-\gamma-1}$. Also, by  Lemma~\ref{lemma:ref}(ii) (the statement on $H_2$), $\|\frac{\partial^2}{\partial s^2}v(u,s))\|_{\cB_\theta}<\infty$ and thus,
$|I_3|\ll \int_Y\tau e^{-u\tau}e^{s\bar\psi}\, d\mu_{\overline{\phi}}=O(1)$. Thus, if $\beta/\gamma\in (2,3)$, then $\frac{\partial}{\partial s}W_1(u,s)= o(u^{\beta-2\gamma-1})$, as $u,s\to 0$.\\

\textbf{ The term  $\bm{W_2(u,s)}$.}

Note that  
\begin{align*}
\frac{\partial}{\partial s}W_2(u,s)=-\int_Y \bar\psi e^{-u\tau}e^{s\bar\psi}\, \frac{\partial}{\partial u}v(u,s)\, d\mu_{\overline{\phi}}
+\int_Y (1-e^{-u\tau}e^{s\bar\psi})\, \frac{\partial^2}{ \partial s\,\partial u}v(u,s)\, d\mu_{\overline{\phi}}.
\end{align*}

\textbf{If $\bm{\beta/\gamma\in (1,2]}$},  by Lemma~\ref{lemma:sm} (statement on $G_{[q_0]}$ with $[q_0]=1$),
$\|\frac{\partial}{\partial u}v(u,s)\|_{\cB_\theta}<\infty$.
Recall $e^{s\bar\psi}\ll e^{sC_0}e^{-s\tau^\gamma}$.
Thus, $\left|\int_Y \bar\psi e^{-u\tau}e^{s\bar\psi}\, \frac{\partial}{\partial u}v(u,s)\, d\mu_{\overline{\phi}}\right|\ll \int_Y \bar\psi\, \frac{\partial}{\partial u}v(u,s)\, d\mu_{\overline{\phi}}=O(1)$.
By Lemma~\ref{lemma:ref}(i) (statement on $K_1$), $\|\frac{\partial^2}{\partial s\, \partial u}v(u,s)\|_{\cB_\theta}\ll u^{\beta-\gamma-1}$.
So, $\left|\int_Y (1-e^{-u\tau}e^{s\bar\psi})\, \frac{\partial^2}{ \partial s\,\partial u} v(u,s)\, d\mu_{\overline{\phi}}\right|\ll u^{\beta-\gamma-1}\int_Y (1- e^{-u\tau}e^{s\bar\psi})\, d\mu_{\overline{\phi}}=o(u^{\beta-\gamma-1})$, as $u,s\to 0$.
Thus, if $\beta/\gamma\in (1,2]$, $\frac{\partial}{\partial s}W_2(u,s)= o(u^{\beta-\gamma-1})$, as $u,s\to 0$.

\textbf{If $\bm{\beta/\gamma\in (2,3)}$}, then we differentiate once more.

\begin{align*}
\frac{\partial^2}{\partial s^2}W_2(u,s)& =-\int_Y \bar\psi^2 e^{-u\tau}e^{s\bar\psi}\, \frac{\partial}{\partial u}v(u,s)\, d\mu_{\overline{\phi}}
-\int_Y \bar\psi e^{-u\tau}e^{s\bar\psi}\, \frac{\partial^2}{\partial s\partial u}v(u,s)\, d\mu_{\overline{\phi}}\\
&\quad-\int_Y \bar\psi e^{-u\tau}e^{s\bar\psi}\, \frac{\partial^2}{\partial u\, \partial s}v(u,s)\, d\mu_{\overline{\phi}}+\int_Y (1-e^{-u\tau}e^{s\bar\psi})\, \frac{\partial^3}{ \partial s^2\,\partial u}v(u,s)\, d\mu_{\overline{\phi}}\\
&=:I_1+I_2+I_3+I_4.
\end{align*}

Since $\bar\psi\in L^2$ for $\beta/\gamma\in (2,3)$, and since $\|\frac{\partial}{\partial u}v(u,s)\|_{\cB_\theta}<\infty$, $|I_1|= O(1)$.
Also, it is easy to see that  $|I_2|= O(1)$ and  $|I_3|= O(1)$.
For $I_4$, we note that by Lemma~\ref{lemma:ref}(ii) (statement on $K_2$), $\|\frac{\partial^3}{ \partial s^2\,\partial u}v(u,s)\|_{\cB_\theta}\ll u^{\beta-2\gamma-1}$.
Thus, $|I_4|=o(u^{\beta-2\gamma-1})$, as $u,s\to 0$. So, if $\beta/\gamma\in (2,3)$, then $\frac{\partial^2}{\partial s^2}W_2(u,s)= o(u^{\beta-2\gamma-1})$, as $u,s\to 0$.

The statements on $\frac{\partial}{\partial s} d(u,s)$ for $\beta/\gamma\in (1,2]$ and on $\frac{\partial^2}{\partial s^2} d(u,s)$ for $\beta/\gamma\in (2,3)$  follow by putting all the above estimates on $W_0, W_1, W_2$ together.\\

\textbf{Proof of item (ii)}.
Recalling~\eqref{evnb} and differentiating in $s$,
\[
 \frac{\partial}{\partial s}\lambda(u,s) =\int_Y\bar\psi\, d\mu_{\overline{\phi}} + \int_Y \bar\psi(e^{s\bar\psi}-1)\, d\mu_{\overline{\phi}}  + e(u,s),
\]
for
\begin{align}\label{derivs}
e(u,s) & = \int_Y \bar\psi e^{s\bar\psi}(1- e^{-u\tau})\, d\mu_{\overline{\phi}} 
\nonumber+\int_Y\bar\psi e^{-u\tau}e^{s\bar\psi}(v(0,0)-v(u,s))\, d\mu_{\overline{\phi}} \nonumber \\
 \nonumber &\quad +\int_Y (1-e^{-u\tau}e^{s\bar\psi})\, \frac{\partial}{\partial s}v(u,s)\, d\mu_{\overline{\phi}}\\
  &=:Z_0(u,s)+Z_1(u,s)+Z_2(u,s).
 \end{align}
 
 A standard calculation  (already used in showing \eqref{eq:utauprec})
 shows that,
 given that $\bar\psi=C_0-C_1\tau^\gamma$ and that  $\mu_Y(\tau\ge x)=c x^{-\beta}(1+o(1))$,
  there exists $C, C'>0$ depending on $c$ and $\beta/\gamma$ so that
 \begin{align}
\int_Y \bar\psi(e^{s\bar\psi}-1)\, d\mu_{\overline{\phi}} =
 \begin{cases}\label{us111}
 h(s)+ h_0(s), &\text{ if }\beta/\gamma\in (1,2],\\
 -s\int_Y \bar\psi^2\, d\mu_{\overline{\phi}}+ Cs^{\beta/\gamma-1}+h_1(s), &\text{ if }\beta\in (2,3),
 \end{cases}
  \end{align} 
  where $h(s)=Cs^{\beta/\gamma-1}$ if  $\beta/\gamma\in (1,2)$, $h(s)=C\log(1/s)$ if  $\beta/\gamma=2$
  and  where $h_0$ and $h_1$ are as follows: (a) $h_0(s)=o(s^{\beta/\gamma-1})$, $h'_0(s)=o(s^{\beta/\gamma-2})$ if $\beta/\gamma\in (1,2)$
 and $h'_0(s)=o(\log(1/s))$ if $\beta/\gamma=2$;
 (b) $h_1(s)=o(s^{\beta/\gamma-1})$, $h'_1(s)=C' s^{\beta/\gamma-2}(1+o(1))$
  and $h''_1(s)=C' s^{\beta/\gamma-3}(1+o(1))$.

  We continue with the study of $e(u,s)$.
  It is easy to see from~\eqref{derivs} with $u=0$ and $s=0$, respectively, that $|e(0,s)|=O(s)$ as $s\to 0$ and that $|e(u,0)|=o(1)$ as $u\to 0$; to show $|e(u,0)|=o(1)$ we also use the Dominated Convergence Theorem. Also, it is easy to see that
  if $\beta/\gamma\in (1,2]$ then
 \begin{align*}
  \left|\frac{\partial}{\partial s}e(0,s)\right|&\ll \|v(0,0)-v(0,s) \|_{\cB_\theta} \int_Y\bar\psi^2 e^{s\bar\psi}\, d\mu_{\overline{\phi}}  
  +\left\|\frac{\partial^2}{\partial s^2}v(0,s)\right\|_{\cB_\theta}\int_Y (1-e^{s\bar\psi})\, d\mu_{\overline{\phi}}\\
  &\ll s\, s^{\beta/\gamma-2} +s^{\beta/\gamma-2} s \int_Y \bar\psi\, d\mu_{\overline{\phi}}\ll s^{\beta/\gamma-1},
    \end{align*}
where in the previous to last inequality we have used Lemma~\ref{lemma:ref} (i) (statement on $\frac{\partial}{\partial s}H_1(0,s)$)
and the estimate in $s$ in \eqref{eq:utaus}. If $\beta/\gamma = 2$, then, again by Lemma~\ref{lemma:ref} (i), the same statement holds with $s^{\beta/\gamma-2}$ replace by $\log 1/s$. In this case,
$\left|\frac{\partial}{\partial s}e(0,s)\right|$ is bounded by $s \log 1/s$.

We continue with the derivatives of  $Z_0, Z_1, Z_2$  in~\eqref{derivs}, when $u\ne 0$,  by considering each of the two cases.\\
  
 \textbf{The term $\bm{Z_0(u,s)}$}. Differentiating in $s$, we obtain
 \begin{align*}
\frac{\partial}{\partial s} Z_0(u,s) =\int_Y \bar\psi^2 e^{s\bar\psi}(1- e^{-u\tau})\, d\mu_{\overline{\phi}}, \quad \frac{\partial^2}{\partial s^2} Z_0(u,s) =\int_Y \bar\psi^3 e^{s\bar\psi}(1- e^{-u\tau})\, d\mu_{\overline{\phi}}. \end{align*}

Using the estimates~\eqref{eq:utaus} in $s$ in~\eqref{eq:utaus},  as $s\to 0$, $\int_Y \bar\psi^2 e^{s\bar\psi}\, d\mu_{\overline{\phi}}=Cs^{\beta/\gamma-2}(1+o(1))$
if $\beta/\gamma\in (1,2)$, $\int_Y \bar\psi^2 e^{s\bar\psi}\, d\mu_{\overline{\phi}}=C\log(1/s)(1+o(1))$
if $\beta/\gamma=2$ and $\int_Y \bar\psi^3 e^{s\bar\psi}\, d\mu_{\overline{\phi}}=Cs^{\beta/\gamma-3}(1+o(1))$
if $\beta/\gamma\in (2,3)$ for some $C>0$ (varying from estimate to estimate).

Thus,  as $u,s\to 0$, $\frac{\partial}{\partial s} Z_0(u,s)=o(s^{\beta/\gamma-2})$, \textbf{if  $\bm{\beta/\gamma\in (1,2)}$}, $\frac{\partial}{\partial s} Z_0(u,s)=o(\log(1/s))$, \textbf{if  $\bm{\beta/\gamma=2}$}
and $\frac{\partial^2}{\partial s^2} Z_0(u,s)=o(s^{\beta/\gamma-3})$, \textbf{if  $\bm{\beta/\gamma\in (2,3)}$}.  \\

  \textbf{The term $\bm{Z_1(u,s)}$}. Differentiating in $s$, we obtain
  
  \begin{align*}
  \frac{\partial}{\partial s} Z_1(u,s) =\int_Y\bar\psi^2 e^{-u\tau}e^{s\bar\psi}(v(0,0)-v(u,s))\, d\mu_{\overline{\phi}}   
  -\int_Y\bar\psi e^{-u\tau}e^{s\bar\psi}\, \frac{\partial}{\partial s} v(u,s)\, d\mu_{\overline{\phi}}.
   \end{align*}
  
  Recall that $\|v(0,0)-v(u,s)\|_{\cB_\theta}\ll u+s$. Thus, \textbf{if  $\bm{\beta/\gamma\in (1,2)}$},   
  \begin{align*}
  \left|\int_Y\bar\psi^2 e^{-u\tau}e^{s\bar\psi}(v(0,0)-v(u,s))\, d\mu_{\overline{\phi}} \right|\ll (u+s)\int_Y\psi^2 e^{s\bar\psi}\, d\mu_{\overline{\phi}}\ll (u+s)s^{\beta/\gamma-2},
  \end{align*}
  where we have used that  $\int_Y \bar\psi^2 e^{s\bar\psi}\, d\mu_{\overline{\phi}}=Cs^{\beta/\gamma-2}(1+o(1))$.

  Recall that by Lemma~\ref{lemma:ref}(i) (statement on $H_1$), $\|\frac{\partial}{\partial s}v(u,s)\|_{\cB_\theta}<\infty$.   Thus, 
  $\left|\int_Y\bar\psi e^{-u\tau}e^{s\bar\psi}\, \frac{\partial}{\partial s} v(u,s))\, d\mu_{\overline{\phi}}\right|=O(1)$.
  Therefore,

  \textbf{If $\bm{\beta/\gamma\in (1,2)}$}, then $\frac{\partial}{\partial s}Z_1(u,s)= O((u +s)s^{\beta/\gamma-2})$.
  
  \textbf{If $\bm{\beta/\gamma=2}$}, then we proceed the same using that  $\int_Y \bar\psi^2 e^{s\bar\psi}\, d\mu_{\overline{\phi}}=C\log(1/s)(1+o(1))$,
  which gives $\frac{\partial}{\partial s}Z_1(u,s)= O((u+s) \log(1/s))$.     
  
  \textbf{If  $\bm{\beta/\gamma\in (2,3)}$},  differentiating once more in $s$ and using a similar argument to the case  $\beta/\gamma\in (1,2)$ above
  (exploiting that $\int_Y \bar\psi^3 e^{s\bar\psi}\, d\mu_{\overline{\phi}}=Cs^{\beta/\gamma-3}(1+o(1))$)  we obtain $\frac{\partial^2}{\partial s^2}Z_1(u,s)= O((u+s) s^{\beta/\gamma-3})$. \\

   \textbf{The term $\bm{Z_2(u,s)}$}. Differentiating in $s$,
  
  \begin{align*}
  \frac{\partial}{\partial s} Z_2(u,s)=-\int_Y \bar\psi e^{-u\tau}e^{s\bar\psi}\, \frac{\partial}{\partial s}v(u,s)\, d\mu_{\overline{\phi}} 
  +\int_Y (1-e^{-u\tau}e^{s\bar\psi})\, \frac{\partial^2}{\partial s^2}v(u,s)\, d\mu_{\overline{\phi}}.
  \end{align*}
  
We already know that  $\|\frac{\partial}{\partial s}v(u,s)\|_{\cB_\theta}<\infty$. Hence, $\left|\int_Y \bar\psi e^{-u\tau}e^{s\bar\psi}\, \frac{\partial}{\partial s}v(u,s)\, d\mu_{\overline{\phi}} \right|=O(1)$.
Also, \textbf{if  $\bm{\beta/\gamma\in (1,2]}$},   by Lemma~\ref{lemma:ref}(i) (statement on $H_1$), $\|\frac{\partial^2}{\partial s^2}v(u,s)\|_{\cB_\theta}\ll u^{\beta-\gamma-1}$. 
Thus, $\left|\int_Y (1-e^{-u\tau}e^{s\bar\psi})\, \frac{\partial^2}{\partial s^2}v(u,s)\, d\mu_{\overline{\phi}} \right|=o(u^{\beta-\gamma-1})$, as $u,s\to 0$.
Thus, if $\beta/\gamma\in (1,2]$, then $\frac{\partial}{\partial s}Z_2(u,s)=o(u^{\beta-\gamma-1})$, as $u,s\to 0$.

 \textbf{If  $\bm{\beta/\gamma\in (2,3)}$},  differentiating once more in $s$ and using a similar argument to the case  $\beta/\gamma\in (1,2]$ above (but using the statement on $H_2$ in Lemma~\ref{lemma:ref}(ii)),
  we obtain $\frac{\partial^2}{\partial s^2}Z_2(u,s)= o(u^{\beta-2\gamma-1})$, as $u,s\to 0$.\\
  
  The statement on $\frac{\partial}{\partial s} e(u,s)$ for $\beta/\gamma\in (1,2]$ and for $\frac{\partial^2}{\partial s^2} e(u,s)$ for $\beta/\gamma\in (2,3)$  follows by putting all the above estimates on $Z_0, Z_1, Z_2$ together.\\
   
  \textbf{Proof of item (iii)}.  We continue from~\eqref{derivu} and compute that
  \begin{align*}
 \kappa(u,s)&=\int_Y\tau \bar\psi e^{-u\tau}e^{s\bar\psi}\, d\mu_{\overline{\phi}}
  -\int_Y\tau\bar\psi  e^{-u\tau}e^{s\bar\psi}(v(0,0)-v(u,s))\, d\mu_{\overline{\phi}}\\
  &\quad +\int_Y \tau \bar\psi e^{-u\tau} e^{s\bar\phi} \frac{\partial}{\partial s} v(u,s) \, d\mu_{\bar\phi} - \int_Y \tau e^{-u\tau} e^{s\bar\phi} \frac{\partial}{\partial u} v(u,s) \, d\mu_{\bar\phi} \\
  &\quad +\int_Y (1-e^{-u\tau}e^{s\bar\psi})\, \frac{\partial}{\partial s}\frac{\partial}{\partial u}v(u,s)\, d\mu_{\overline{\phi}}
   \end{align*}
   and 
 \begin{align}\label{kappadiff}
 \frac{\partial}{\partial s}\kappa(u,s)&=\int_Y\tau \bar\psi^2 e^{-u\tau}e^{s\bar\psi}\, d\mu_{\overline{\phi}}
  -\int_Y\tau\bar\psi^2  e^{-u\tau}e^{s\bar\psi}(v(0,0)-v(u,s))\, d\mu_{\overline{\phi}}\\
 \nonumber  & \quad + 2\int_Y\tau\bar\psi  e^{-u\tau}e^{s\bar\psi}\frac{\partial}{\partial s}v(u,s)\, d\mu_{\overline{\phi}} - 2\int_Y \bar\psi e^{-u\tau}e^{s\bar\psi}\, \frac{\partial}{\partial s}\nonumber \frac{\partial}{\partial u}v(u,s)\, d\mu_{\overline{\phi}}\\
 \nonumber & \quad + \int_Y \tau e^{-u\tau} e^{s\bar\phi} \frac{\partial^2}{\partial s^2} v(u,s) \, d\mu_{\bar\phi}
 - \int_Y \tau^2 e^{-u\tau} e^{s\bar\phi} \frac{\partial}{\partial u} v(u,s) \, d\mu_{\bar\phi} \\
  \nonumber & \quad +\int_Y (1-e^{-u\tau}e^{s\bar\psi})\, \frac{\partial^2}{\partial s^2}\frac{\partial}{\partial u}v(u,s)\, d\mu_{\overline{\phi}}\\
  \nonumber & =:\kappa_1(u,s)+\kappa_2(u,s)+\kappa_3(u,s)  
  +\kappa_4(u,s) +\kappa_5(u,s) +\kappa_6(u,s) +\kappa_7(u,s). \end{align}
   We provide the argument for the case  $\beta/\gamma\in (1,2]$ .  The case  $\beta/\gamma\in (2,3)$ follows by a similar argument after differentiating~\eqref{kappadiff} once more in $s$.
   
   Using Lemma~\ref{lemma:ref} (i),
   \begin{align*}
    \kappa(u,s)&=\int_Y\tau \bar\psi e^{-u\tau}e^{s\bar\psi}\, d\mu_{\overline{\phi}}
    +O\left((u+s)\int_Y\tau\bar\psi  e^{-u\tau}e^{s\bar\psi}\, d\mu_{\overline{\phi}}\right)   
    +O\left(u^{\beta-\gamma-1} (u+s)\right).
       \end{align*}
    Taking $s=0$ in this equation, we get that there exists $C>0$ so that
   \begin{align*}
   \kappa(u,0)&=\int_Y\tau \bar\psi e^{-u\tau}\, d\mu_{\overline{\phi}}  +O\left(u\int_Y\tau\bar\psi  e^{-u\tau}\, d\mu_{\overline{\phi}}\right)  +O\left(u^{\beta-\gamma} \right)\\
   &=C u^{\beta-\gamma-1}(1+o(1)),
     \end{align*}
  where in the last equality we have used~\eqref{eq:utauprec}.
  
  We estimate $\kappa_1,\ldots,\kappa_7$ in~\eqref{kappadiff}.
  Note that differentiating once more in~\eqref{us11} and using the estimates
  in Section~\ref{sec:integr},
   $\int_Y\tau \bar\psi^2 e^{-u\tau}\, d\mu_{\overline{\phi}}=Cu^{\beta-2\gamma-1}(1+o(1))$.
  Thus, as $u,s\to 0$,
     \begin{align*}
  \kappa_1(u,s)=\int_Y\tau \bar\psi^2 e^{-u\tau}\, d\mu_{\overline{\phi}} +\int_Y\tau \bar\psi^2 (e^{s\bar\psi}-1)e^{-u\tau}\, d\mu_{\overline{\phi}}  
  =Cu^{\beta-2\gamma-1}(1+o(1)).    
  \end{align*}
  By arguments already used in estimating quantities in proof of items (i) and (ii) above, $\kappa_2(u,s), \kappa_3(u,s),\kappa_4(u,s),\kappa_6(u,s)=o(u^{\beta-2\gamma-1})$, as $u,s\to 0$.
  Finally, by Lemma~\ref{lemma:ref}(i) (statement on $K_2$),   $\|\frac{\partial^2}{\partial s^2}\frac{\partial}{\partial u}v(u,s)\|_{\cB_\theta}\ll u^{\beta-2\gamma-1}$.
  Thus, $\kappa_5(u,s),\kappa_7(u,s)=o(u^{\beta-2\gamma-1})$, as $u,s\to 0$.~\end{pfof}

\section{Proof of Proposition~\ref{prop:gmsec}}

Using the technical results in Section~\ref{sec:refined} we can proceed to the proof of
 Proposition~\ref{prop:gmsec}. We recall that this is a refined version
of Proposition~\ref{prop:gm} under somewhat stronger assumptions (that is, regular variation of the tail behaviour). In this sense, the task of this section 
is to go over the steps of the proof of Proposition~\ref{prop:gm} and obtain higher order expansions.
From this proof, we recall that a first step is to refine the estimate 
on $\frac{\partial}{\partial s}\frac{\partial}{\partial u}\bar p(u,s)$
(see~\eqref{eq:dounderpbar}). For the proof of Proposition~\ref{prop:gmsec}, we shall need to understand $\frac{\partial^2}{\partial s^2}\frac{\partial}{\partial u}\bar p(u,s)$ as $u,s\to 0$.

\begin{lemma}\label{lemma:doubderder} Assume the setup of Proposition~\ref{prop:gmsec} with larger range of $\gamma$,
namely $\gamma\in (\beta-1,\beta)$. There exist $C_2, C_3, C_4, C_5>0$(varying from line to line) so that the following hold as $u,s\to 0$.
\begin{itemize}
\item[(i)] If $\beta/\gamma\in (1,2)$ then 
\(
\frac{\partial^2}{\partial s^2}\frac{\partial}{\partial u}\bar p(u,s)=-C_2s^{\beta/\gamma-2}(1+o(1)+C_3 u^{\beta-2\gamma-1}(1+o(1)).\)
Also, $\frac{\partial}{\partial s}\frac{\partial}{\partial u}\bar p(u,s)= C_4u^{\beta-\gamma-1}(1+o(1))+C_5 s u^{\beta-2\gamma-1}(1+o(1))$.

\item[(ii)] If $\beta/\gamma=2$ then 
\(
\frac{\partial^2}{\partial s^2}\frac{\partial}{\partial u}\bar p(u,s)=-C_2\log(1/s)(1+o(1))+C_3 u^{-1}(1+o(1)).\)
Also, $\frac{\partial}{\partial s}\frac{\partial}{\partial u}\bar p(u,s)= C_4u^{\beta-\gamma-1}(1+o(1))+ C_3s u^{-1}(1+o(1))-C_2 s\log(1/s)(1+o(1))$.
\item[(ii)] If $\beta/\gamma\in (2,3)$ then 
\(
\frac{\partial^3}{\partial s^3}\frac{\partial}{\partial u}\bar p(u,s)=-C_2s^{\beta/\gamma-3}(1+o(1))-C_3 u^{\beta-3\gamma-1}(1+o(1)).\)
Also, $\frac{\partial^2}{\partial s^2}\frac{\partial}{\partial u}\bar p(u,s)\Big|_{s=0}=- C_4u^{\beta-2\gamma-1}(1+o(1))+C_3 s u^{\beta-3\gamma-1}(1+o(1))$.
\end{itemize}

\end{lemma}
\begin{proof} First we recall from~\eqref{eq:dounderpbar} that
\begin{align*}
 \frac{\partial}{\partial s}\frac{\partial}{\partial u}\bar p(u,s)
  &=  -\frac{\frac{\partial}{\partial u}\lambda(u,s)\frac{\partial}{\partial s}\lambda(u,s)}{\lambda(u,s)^2}+\frac{\frac{\partial}{\partial s}\frac{\partial}{\partial u}\lambda(u,s)}{\lambda(u,s)}.
\end{align*}
Set $A(u,s):=\frac{\partial}{\partial u}\lambda(u,s)\frac{\partial}{\partial s}\lambda(u,s)$ and recall (for instance, from Lemma~\ref{lemm:refev}(iii)) that 
$\kappa(u,s)=\frac{\partial}{\partial s}\frac{\partial}{\partial u}\lambda(u,s)$.
Compute that
\begin{align*}
 \frac{\partial^2}{\partial s^2}\frac{\partial}{\partial u}\bar p(u,s)
  &=  -\frac{\frac{\partial}{\partial s}A(u,s)}{\lambda(u,s)^2}-2\frac{A(u,s)\frac{\partial}{\partial s}\lambda(u,s) }{\lambda(u,s)^3}
  +\frac{\frac{\partial}{\partial s}\kappa(u,s)}{\lambda(u,s)}-\frac{\kappa(u,s)}{\lambda(u,s)^2}\\
  &=:N_1(u,s)+N_2(u,s)+N_3(u,s)+N_4(u,s).
\end{align*}

We provide the proof of item (i). Item (ii) follows by the same argument using the statements for the case $\beta/\gamma=2$ in Lemma~\ref{lemm:refev} (i) and (ii). Item (iii) follows by a similar argument after differentiating once more and using the statements for the case $\beta/\gamma\in (2,3)$ in Lemma~\ref{lemm:refev} (i) and (ii). 

From the estimates of Lemma~\ref{lemm:refev} (i) and (ii) (the statements for the case $\beta/\gamma\in (1,2)$), it is easy to see that $N_2$ and $N_4$
do not contribute to the main asymptotics (because they go to a constant as $u,s\to 0$). We need to look at  $N_1$ and $N_3$.\\

\textbf{The term $N_1(u,s)$}.
Using the same notation as in Lemma~\ref{lemm:refev} (i) and (ii),
\begin{align*}
A(u,s)=
\left(-\tau^*+d(u,s)\right)\left(\bar\psi^*+ C s^{\beta/\gamma-1}+h(s)+h_0(s) +e(u,s) \right)
\end{align*}
and 
\begin{align*}
\frac{\partial}{\partial s} A(u,s)=&
\frac{\partial}{\partial s} d(u,s)\, \left(\bar\psi^*+ C s^{\beta/\gamma-1}+h(s)+h_0(s)  +e(u,s) \right)\\
&+ \left(-\tau^*+d(u,s)\right)\left( C(\beta/\gamma-1) s^{\beta/\gamma-2}+\frac{\partial}{\partial s} h(s)+h_0'(s)  +\frac{\partial}{\partial s} e(u,s) \right).\end{align*}

Using all the estimates on $d, h_0, e$ in Lemma~\ref{lemm:refev} (i) and (ii) (the statements for the case $\beta/\gamma\in (1,2)$),
we obtain that there exist $C_2, C_2'>0$ so that
\begin{align*}
\frac{\partial}{\partial s} A(u,s)&=-C_2 s^{\beta/\gamma-2}(1+o(1)) + C_2' u^{\beta-\gamma-1}(1+o(1)),
\end{align*}
which gives the asymptotics for $N_1(u,s)$.
In the previous displayed equation, apart from the estimates on $\frac{\partial}{\partial s} d(u,s)$ and  $\frac{\partial}{\partial s} e(u,s)$, we have used the immediate consequence of 
 Lemma~\ref{lemm:refev}(ii) that $d(u,s)=O(s u^{\beta-\gamma-1})$ and that $ e(u,s)= o(s u^{\beta-\gamma-1})$.
 
 \textbf{The term $N_3(u,s)$}. By Lemma~\ref{lemm:refev} (iii) (the statement for the case $\beta/\gamma\in (1,2)$),
 $\frac{\partial}{\partial s}\kappa(u,s)=C_3 u^{\beta-2\gamma-1}(1+o(1))$, for some $C_3>0$. This gives the same asymptotics for $N_3$. Therefore,
\begin{align*}
N_1(u,s)+N_3(u,s)=-C_2 s^{\beta/\gamma-2}(1+o(1))+C_3 u^{\beta-2\gamma-1}(1+o(1)),\end{align*}
which gives the first statement in item (i).

The second statement in item (i) follows immediately from the first together with the asymptotics of $\kappa(u,0)$ in Lemma~\ref{lemm:refev} (iii).~\end{proof}

We can now proceed to

\begin{pfof}{Proposition~\ref{prop:gmsec}}
We redo all steps in the proof of  Proposition~\ref{prop:gm}(ii) using  Lemma~\ref{lemm:refev}.

Recall $\bar p(u,s)=\log \lambda(u,s)$. The analogue of~\eqref{eq:expld1} is

\begin{equation}
  \frac{\partial}{\partial u}\bar p(u,s)=-\tau^*+D(u,s),\quad \frac{\partial}{\partial s}\bar p(u,s)= \bar\psi^*+E(u,s),
 \end{equation}
 where
 \begin{itemize}
  \item[(a)] $D(u,s)$ satisfies the same properties as $d(u,s)$ in Lemma~\ref{lemm:refev}(i).
    \item[(b)]$E(u,s)$ satisfies the same properties as $e(u,s)$ in Lemma~\ref{lemm:refev}(ii). 
    \end{itemize}
 By Lemma~\ref{lemm:refev}(i) and (ii), we have the following refined version of~\eqref{eq:derpos}
 (with $C$ varying from line to line).
   \begin{eqnarray}\label{eq:derpos2}
 \frac{\partial^2}{\partial s^2}\bar p(0,s)=Cs^{\beta/\gamma-2}(1+o(1)),& &\text{ if }\beta/\gamma\in (1,2)\\[1mm]
 \nonumber    \frac{\partial^2}{\partial s^2}\bar p(0,s)=C\log(1/s)(1+o(1)), & & \text{ if }\beta/\gamma=2 \\[1mm]
  \nonumber \frac{\partial^3}{\partial s^3}\bar p(0,s)=Cs^{\beta/\gamma-3}(1+o(1)),  & &\text{ if }\beta/\gamma\in (2,3).
   \end{eqnarray}
The analogue of~\eqref{eq:impl}  for any small $u_0>0$ is
 \begin{align*}
  \bar p(u_0,s)-\bar p(0,s)=-\tau^* u_0+ \int_0^{u_0} D(u,s)\, du:=-\tau^* u_0+ L(u_0,s),
  \end{align*}
where $D(u,s)$ satisfies the same properties as $d(u,s)$ in Lemma~\ref{lemm:refev}(i). 
Moreover, as in the proof of Proposition~\ref{prop:gm} (ii),
\begin{align}\label{derD}
\frac{\partial}{\partial s}D(u,s)=\frac{\partial}{\partial s}\frac{\partial}{\partial u}\bar p(u,s).
\end{align}
By the argument used in the proof of Proposition~\ref{prop:gm} in deriving~\eqref{eq:derivp},
\begin{align}\label{eq:derivp2}
u_0'(s) = \frac{\frac{\partial}{\partial s} M(u_0,s)}{\tau^*-\frac{\partial}{\partial u_0} M(u_0,s)},
\end{align}
where, as in the proof of Proposition~\ref{prop:gm},
\begin{align}\label{eq;m0}
 M(u_0,s)=L(u_0,s)+\bar p(0,s) \qquad \text{ with }\ \frac{\partial}{\partial u_0} L(u_0,s)=D(u_0,s).
 \end{align}
 Differentiating~\eqref{eq:derivp2} once more in $s$,
   \begin{align}\label{secondp}
\nonumber p''(s)& =\frac{\frac{\partial^2}{\partial s^2} M(u_0,s)(\tau^*-\frac{\partial}{\partial u_0} M(u_0,s)) } {\left(\tau^*-\frac{\partial}{\partial u_0} M(u_0,s)\right)^2}
 +\frac{\frac{\partial}{\partial s} M(u_0,s)\frac{\partial^2}{\partial u_0\partial s} M(u_0,s)}{\left(\tau^*-\frac{\partial}{\partial u_0} M(u_0,s)\right)^2}\\
 &=:M_1(u_0,s)+M_2(u_0,s).
\end{align}

We complete the proof of (i), that is, we treat the case $\beta/\gamma\in (1,2)$ using the estimates in Lemma~\ref{lemm:refev}.
The precise asymptotics in (ii) for the case $\beta/\gamma = 2$ follow by the same argument using the corrsponding estimates in Lemma~\ref{lemm:refev}.
Item (iii), the case $\beta/\gamma\in (2,3)$, (after taking one more derivative in $s$)
is similar and omitted.\\

\textbf{Proof of (i), the case $\bm{\beta/\gamma\in (1,2)}$}.\\

\textbf{The term $\bm{M_1(u_0,s)}$ defined in~\eqref{secondp}}.
Differentiating~\eqref{eq;m0},
\begin{align}\label{eq:m1111}
 \frac{\partial}{\partial s} M(u_0,s)=\frac{\partial}{\partial s} L(u_0,s)+\frac{\partial}{\partial s}\bar p(0,s).
 \end{align}
 
 Using~\eqref{eq:expld2},~\eqref{derD}
and Lemma~\ref{lemma:doubderder}(i),
\begin{align*}
 \frac{\partial}{\partial s} L(u_0,s)&=\int_0^{u_0}\frac{\partial}{\partial s}D(u,s)\, du
 =  C_4 u_0^{\beta-\gamma}(1+o(1))+ C_3 s\, u_0^{\beta-2\gamma}(1+o(1)).
\end{align*}
By Proposition~\ref{prop:gm} (ii), $p(s)=u_0(s)=\frac{\bar p(0,s)}{\tau^*}=s \frac{\bar\psi^*}{\tau^*}(1+o(1))$, as $s\to 0$. Thus,
\begin{align*}
 \frac{\partial}{\partial s} L(u_0,s)&=C_4 s^{\beta-\gamma}(1+o(1))+ C_3 s^{\beta-2\gamma+1}(1+o(1))=C_4 s^{\beta-\gamma}(1+o(1)),
 \end{align*}
 where in the last equality we have used that $\gamma<1$.

By Lemma~\ref{lemm:refev}(ii), $\frac{\partial}{\partial s}\bar p(0,s)=\bar\psi^*+Cs^{\beta/\gamma-1}(1+o(1))$.
Since $\beta>\gamma$,
\begin{align}\label{firstderivp}
 \frac{\partial}{\partial s} M(u_0,s) =\bar\psi^*+Cs^{\beta/\gamma-1}(1+o(1))=\bar\psi^*(1+o(1)).
 \end{align}
  Differentiating~\eqref{eq:m1111} once more in $s$ and using~\eqref{eq:derpos2},
\begin{align*} \frac{\partial^2}{\partial s^2} M(u_0,s)=\frac{\partial^2}{\partial s^2} L(u_0,s)+\frac{\partial^2}{\partial s^2}\bar p(0,s)
=\frac{\partial^2}{\partial s^2} L(u_0,s)+Cs^{\beta/\gamma-2}(1+o(1)).
 \end{align*}
 
 Next, recall~\eqref{derD} and note that
$\frac{\partial^2}{\partial s^2}D(u,s)=\frac{\partial^2}{\partial s^2}\frac{\partial}{\partial u}\bar p(u,s)$.
By Lemma~\ref{lemma:doubderder}(i),
\(
\frac{\partial^2}{\partial s^2}\frac{\partial}{\partial u}\bar p(u,s)=-C_2s^{\beta/\gamma-2}(1+o(1)+C_3 u^{\beta-2\gamma-1}(1+o(1)).\)
Also, recall that $u_0(s)=s \frac{\bar\psi^*}{\tau^*}(1+o(1))$, as $s\to 0$ for $C_2, C_3>0$.
 Thus, 
  \begin{align}\label{eq:reftoit}
  \frac{\partial^2}{\partial s^2} L(u_0,s)&=\int_0^{u_0}\frac{\partial^2}{\partial s^2}D(u,s)\, du
  =-C_2s^{\beta/\gamma-1}(1+o(1)+C_3 s^{\beta-2\gamma}(1+o(1))\nonumber\\
  &=C_3 s^{\beta-2\gamma}(1+o(1)).
  \end{align}
  
  \begin{rmk}\label{rmk:reftoit}
   If we do not assume regular variation for the tail
   $\mu_{\overline\phi}(\psi_0\ge x)=\mu_{\overline\phi}(\tau^\gamma\ge x)$, we can still use the same steps as above in obtaining~\eqref{eq:reftoit}
   and rougher calculations, similar to the ones used in obtaining~\eqref{eq:utau}, to show that $\left|\frac{\partial^2}{\partial s^2} L(u_0,s)\right|=O(s^{\gamma(\beta/\gamma-2)})$. In particular, following these steps one has that if $\psi_0\in L^{q_1}(\mu_{\overline\phi})$ then $\left|\frac{\partial^2}{\partial s^2} L(u_0,s)\right|=O(s^{a(q_1-2)})$ for some $a>0$,
   so $\left|\frac{\partial^2}{\partial s^2} L(u_0,s)\right|=o(1)$, as $s\to 0$.
  \end{rmk}

    Putting the previous three displayed equations together and noticing that $s^{\beta-2\gamma}>s^{\beta/\gamma-2}$ (since $\gamma<1$), we obtain
   \begin{align}\label{secderivp}
   \frac{\partial^2}{\partial s^2} M(u_0,s)=C_3 s^{\beta-2\gamma}(1+o(1)).
   \end{align} 
   
 We have  $\frac{1}{\tau^*-\frac{\partial}{\partial u_0} M(u_0,s)}=\frac{1}{\tau^*-D(u_0,s)}
  =\frac{1}{\tau^*}\left(1+O(D(u_0,s)\right)$
 as in the proof of Proposition~\ref{prop:gm} (ii).
  Using the properties of $D_0(u,s)$ in item (a) after~\eqref{eq:expld1}  (both smoothness in $s$ and asymptotics of $D(u_0, 0)$),
  and using that $u_0(s)=O(s)$, we have
  \[
  \frac{1}{\tau^*-\frac{\partial}{\partial u_0} M(u_0,s)}=\frac{1}{\tau^*}(1+O(s^{\beta-1}+s))
  \] as $s\to 0$.
  This together with~\eqref{secderivp}  gives that as $s\to 0$,
  \begin{align}\label{eqM1}
  M_1(u_0,s)=C_3 s^{\beta-2\gamma}(1+o(1)).  \end{align}
   
 \textbf{The term $\bm{M_2(u_0,s)}$ defined in~\eqref{secondp}}.

 Differentiating~\eqref{eq:m1111} once more in $u_0$,
  $\frac{\partial^2}{\partial u_0\partial s} M(u_0,s)=\frac{\partial^2}{\partial u_0\partial s} L(u_0,s)$.
  Recall that $ \frac{\partial}{\partial s} L(u_0,s)=\int_0^{u_0}\frac{\partial}{\partial s}D(u,s)\, du$
  and that $D(u,s)$ is uniformly continuous in $u$. Thus, $\frac{\partial^2}{\partial u_0\partial s} M(u_0,s)= \frac{\partial}{\partial s}D(u_0,s)$.
  Recalling~\eqref{derD},
  \begin{align*}
  \frac{\partial^2}{\partial u_0\partial s} M(u_0,s)=\frac{\partial}{\partial s}\frac{\partial}{\partial u}\bar p(u,s)\Big|_{u=u_0}.  
    \end{align*}  
 By Lemma~\ref{lemma:doubderder}(i),
\begin{align*}
 \frac{\partial}{\partial s}\frac{\partial}{\partial u}\bar p(u,s)\Big|_{u=u_0} =  C_4 u_0^{\beta-\gamma-1}(1+o(1))+ C_3 s\, u_0^{\beta-2\gamma-1}(1+o(1)).
\end{align*}  for $C_3, C_4>0$.
   Since $u_0(s)=s \frac{\bar\psi^*}{\tau^*}(1+o(1))$, as $s\to 0$,   
  \begin{align*}
  \frac{\partial^2}{\partial u_0\partial s} M(u_0,s)=C_4 s^{\beta-\gamma-1}(1+o(1))+ C_3 s^{\beta-2\gamma}(1+o(1))=C_4 s^{\beta-\gamma-1}(1+o(1)),
     \end{align*}  
  where in the last equation we have used again that $\gamma<1$.
  
  Recalling~\eqref{firstderivp} and that $\frac{1}{\tau^*-\frac{\partial}{\partial u_0} M(u_0,s)}=
  =\frac{1}{\tau^*}(1+o(1))$,
  we have $M_2(u_0,s)= \frac{\bar\psi^*}{\tau^*}C_4 s^{\beta-\gamma-1}(1+o(1))$.  This together with~\eqref{eqM1}
  gives the conclusion after recalling again that $\gamma<1$, which ensures that   $s^{\beta-\gamma-1} > s^{\beta-2\gamma}$.   ~\end{pfof}

\section{Proofs of the main abstract results}
\label{sec:proofmainth}

The proofs of the main results will make use of the restricted pressure. 
Analogous to \cite[Definition 5.1]{RuhSar22}, we define 
\begin{align*}
q(a) = q_{\phi, \psi}(a) & := \sup\left\{P_{\F, \nu}(\phi):\nu\in\mathcal{M}_{\F},\int_{Y^\tau}\psi\, d\nu=a \right\} \\
&\ =  \sup\left\{\frac{P_{T, \mu}(\bar\phi)}{\int\tau~d\mu}:\mu\in\mathcal{M}_T(\tau),\frac{\int_{Y}\bar\psi\, d\mu}{\int\tau~d\mu}=a \right\}.
\end{align*}

\subsection{Proof of Theorem~\ref{thm:firstmain}}

\begin{pfof}{Theorem~\ref{thm:firstmain}} Given Proposition~\ref{prop:gm} with $q_1>3$, the details are very similar with
those in~\cite[Proof of  Lemma 5.2]{RuhSar22} (and also the main line of the argument in~\cite[Proof of  Proposition 6.1]{RuhSar22}).
We recall most of the details, partly for completeness, partly because our setup is different (unbounded potential but more restricted $\psi$).

By Proposition~\ref{prop:gm}(ii), $p'(0)=\frac{\int_Y\overline{\psi}\, d\mu_{\bar\phi}}{\int_Y\tau\, d\mu_{\bar\phi}}=\int_{Y^\tau}\psi\, d\nu_{\phi}=a_0$. By assumption, $\nu_{\phi}$ is the unique equilibrium measure for $\phi$. 
Since $p''(s) \geq 0$ is continuous with $p''(0) = \sigma^2 > 0$
(by Proposition~\ref{prop:gm}), $p'$ is strictly increasing near $0$.

Given $h\in (0, \delta_0)$, for $\delta_0$ is as in Proposition~\ref{prop:gm},  let $a\in (p'(0), p'(h))$. By the Intermediate Value Theorem, there exists $s\in (0,h)$ so that $p'(s)=a$.
By Proposition~\ref{prop:gm}(ii) and (iii), the second derivative is well-defined whenever $q_1>2$.

We next show that $p$ is strictly convex in our domain of interest.  Throughout the rest of the proof let $K > \sigma^2$, so $\delta_0\frac{\sigma^2}{K}<\delta_0$.
By the assumption $q_1 > 3$, the third derivative $p'''$ is well-defined and we can assume $|p'''|<K$
by taking $K$ larger if necessary.
We use this to show strict convexity and that the solution to the equation $p'(s)=a$ in $s$ is unique.
To see this, we recall  the argument by contradiction in~\cite[Proof of Lemma 5.2]{RuhSar22}.
As in~\cite[Proof of Lemma 5.2]{RuhSar22}, if there exists $s_0\ne s, s\in \left(0,\delta_0\frac{\sigma^2}{K}\right)$ 
so that $p'(s_0)=a$ then $p''$ would have vanished in this interval.
This is not possible because for some $s'\in (0,s)$,
\[
 |p''(s)-\sigma^2|=|p''(s)-p''(0)|= s |p'''(s')|\le 
 K \cdot \left(\delta_0\frac{\sigma^2}{K}\right)=\delta_0\sigma^2\ne 0.
\] 

Next we find useful relationships between $a$, $s$ and $\nu_s$ for the appropriate $s$.
For the unique $s$ so that $p'(s)=a$, we know that $R(u,s)$ satisfies the spectral gap: this follows since $R(0,0)$ has a spectral gap in $\cB$ and $R(u,s)$ is continuous in $u, s$ (by Lemma~\ref{lemma:sm}). Thus, the potential $\overline{\phi+s\psi-p(s)}$ has a unique equilibrium state ${\mu}_s$.  This projects to an equilibrium state $\nu_s$ for the potential $\phi+s\psi$ (the unique such measure), as follows. 
First note that from the Gibbs property and since $s, p(s)>0$ and $\bar\psi<\infty$, we get
$$
\int\tau~d\mu_s\ll \int \tau e^{s\bar\psi-\tau p(s) }~d\mu_{\bar\phi} \ll \int\tau~d\mu_{\bar\phi}<\infty,
$$
so $\mu_s\in \M_T(\tau)$ and we obtain $\nu_s\in \M_F$ from \eqref{eq:AK}.  Moreover by the Abramov formula,
$P_{F, \nu_s}(\phi+s\phi-p(s)) =0$, which firstly implies that $\nu_s$ is an equilibrium state for $\phi+s\psi$.  It is also standard to show that this is the unique equilibrium state for this potential and that $\int \psi~d\nu_s = p'(s) =a$, as above.  Moreover, 
if $\nu\in \M_F$ has $P_\nu(\phi)>P_{\nu_s}(\phi)$ and $\int\psi~d\nu=a$, then
$$P_\nu(\phi+s\psi) = P_\nu(\phi)+sa>P_{\nu_s}(\phi)+sa = P_{\nu_s}(\phi+s\psi)=p(s),$$
a contradiction.  Therefore, 
\begin{equation} P_\nu(\phi)\le  p(s) - s\int\psi~d\nu_s= P_{\nu_s}(\phi)=q(a)\label{eq:mu_s}
\end{equation}
for any $\nu\in \M_F$ with $\int\psi~d\nu=a$.

The final task here is to get a relation between $a-a_0$ in terms of $P(\phi)-P_{\nu}(\phi)$.
Recall $q_1 > 3$. By Proposition~\ref{prop:gm}(ii), $p'''$ is $C^{q_1-[q_1]}$.
Thus,
\[
p(s) = p(0)+ sp'(0)+ \frac{s^2}{2}p''(0)+ \frac{s^3}{6}p'''(0)+ O(s^{3+\eps}),
\]
for some $\eps>0$, so $p'(s) = p'(0)+ sp''(0)+\frac{s^2}{2}p'''(0)+O(s^{2+\eps})$.
Then for $s$ so that $p'(s)=a$, and recalling that $p''(0)=\sigma^2$,
\begin{align*}
a-a_0&=p'(s)-p'(0)=sp''(0)+\frac{s^2}{2}p'''(0)+O(s^{2+\eps})\\
&= s\left(p''(0)+\frac{s}{2}p'''(0)+O(s^{1+\eps})\right)
=s\sigma^2\left(1+O(s\sigma^{-2})\right),
\end{align*}
where in the last step we have used that $s\in\left(0,\delta_0\frac{\sigma^2}{K}\right)$ and that $|p'''(0)|<K$. 
Hence,
\begin{align}\label{eq:expssss}
s=\frac{a-a_0}{\sigma^2}\left(1+O\left(s^{1+\eps}\sigma^{-2}\right)\right).
\end{align}

Next, arguing word for word as in the \cite[Proof of Lemma 5.2, item (4)]{RuhSar22}, $q(a_0)=P_{\nu_\phi}(\phi)$
and since, by assumption, $P_{\nu_\phi}(\phi)=p(0)=0$, we have $q(a_0)=0$. This together with~\eqref{eq:mu_s}, the fact
that $a=p'(s)$, the expansions of $p(s)$ and $p'(s)$ and \eqref{eq:mu_s}, imply that for some $\eps>0$,
\begin{align*}
q(a_0)-q(a)&=sp'(s)-p(s) = \frac{s^2}{2}\sigma^2+ \frac{s^3}{6}p'''(0)+O(s^{3+\eps}).
\end{align*}
This together with~\eqref{eq:expssss} gives 
\begin{align*}
q(a_0)-q(a)=\frac{(a-a_0)^2}{2\sigma^2}\left(1+O\left(\sigma^{-2}(a-a_0)\right)\right).
\end{align*}
So for $\nu\in \M_\F$ with $\int\psi~d\nu=a$, the above equation and~\eqref{eq:mu_s} imply
\begin{equation}
P_{\nu_\phi}(\phi)- P_\nu(\phi) \ge P_{\nu_\phi}(\phi)- P_{\nu_s}(\phi) = \frac{(a-a_0)^2}{2\sigma^2} \left(1+O\left(\sigma^{-2}(a-a_0)\right)\right).
\label{eq:near_final}
\end{equation}

Making $a-a_0 = \int\psi~d\nu- \int\psi~d\nu_\phi$ subject of this equation
gives
\[
 \int\psi~d\nu- \int\psi~d\nu_\phi \le C_{\phi, \psi}\sqrt2 \sigma\sqrt{P_{\nu_\phi}(\phi)- P_\nu(\phi)},
\]
where the constant $C_{\phi, \psi} \geq 1$ tends to $1$ as $\int \psi \, d\nu \to \int \psi \, d\nu_{\phi}$.
Continuing with  $\nu_s$, the equilibrium state of $\phi + s\psi$, we get the more precise form
\[
 \int\psi~d\nu_s - \int\psi~d\nu_\phi = \sqrt2 \sigma\sqrt{ P_{\nu_\phi}(\phi)- P_{\nu_s}(\phi) } + O\left(P_{\nu_\phi}(\phi)- P_{\nu_s}(\phi)\right).
\]
which can be rewritten as~\eqref{eq:ho} as required.
~\end{pfof}

\subsection{Proof of Theorem~\ref{thm:secmain}}

We shall need the following fact, which relies on the positivity
of $p''(s)$ given by Proposition~\ref{prop:gmsec}.

\begin{lemma}\label{trivial} Take $\beta/\gamma\in (1,3)$ and  $a \in (p'(0), p'(\delta_0))$, where $\delta_0$ is as in Proposition~\ref{prop:gm}.
Then $p''(s)>0$ for $s\in (0,\delta_0)$ and there exists a unique $s\in (0,\delta_0)$ satisfying $p'(s) = a$.
\end{lemma}

\begin{proof}
By Proposition~\ref{prop:gmsec}, both for $\beta/\gamma\in (1,2]$ and for $\beta/\gamma\in (2,3)$,
the first derivative $p'$ is bounded. 
For $\beta/\gamma\in (1,2)$, the positivity of $p''(s)$ is given by Proposition~\ref{prop:gmsec} (i).
For the case $\beta/\gamma\in (2,3)$,  Proposition~\ref{prop:gm}(iii) ensures that $p''(0)=\sigma^2$. This together
with Proposition~\ref{prop:gmsec} (ii) gives the positivity of $p''(s)$ when $\beta/\gamma\in (2,3)$.
It follows that $p'$ is a strictly increasing function and the conclusion follows.
\end{proof}

\begin{pfof}{Theorem~\ref{thm:secmain}}
 Let $a_0=\int\psi~d\nu_\phi$ and $a=\int\psi~d\nu$ and assume $a>a_0$.
 By Lemma~\ref{trivial}, $p'(s)=a$ has a unique solution.
This allows us to repeat the argument recalled in obtaining~\eqref{eq:mu_s} and to obtain $q(a)=p(s)-sa$.
As in the proof of Theorem~\ref{thm:firstmain}, recall that $q(a_0)=P_{\nu_\phi}(\phi)$
and $q(a)=P_{\nu_s}(\phi)$, where $\nu_s$ is the unique equilibrium measure for $\psi+s\psi$.
 Let $\nu$ be any $\F$-invariant probability measure so that
$a=\int_{Y^\tau}\psi\, d\nu>a_0=\int_{Y^\tau}\psi\, d\nu_{\phi}$.\\
  
 \textbf{Proof of item (a), the case $\bm{\beta/\gamma\in(1,2]}$.}
 Note that $a-a_0=p'(s)-p'(0)$. Using   Proposition~\ref{prop:gmsec}(i),
\begin{align*}
a-a_0= s p''(s) (1+o(1))=C_2\, s\, s^{\beta-\gamma-1}(1+o(1))=C_2s^{\beta-\gamma}(1+o(1)),
\end{align*}
and so, 
\begin{align}\label{aa000}
s=\left(\frac{a-a_0}{C_2}\right)^{1/(\beta-\gamma)}(1+o(1)).
\end{align}Since $q(a_0)=0$, $q(a_0)-q(a)  = sp'(s)-p(s)$.
The Taylor expansion with remainder gives $p(y)=p(x)+p'(x)(y-x)+\int_x^y (y-\xi) p''(\xi)\, d\xi$.
Taking $y=0$ and $x=s$,
\(
q(a_0)-q(a) =sp'(s)-p(s)=\int_0^s \xi p''(\xi)\, d\xi.
\)
By Proposition~\ref{prop:gmsec}(i), we have
\begin{align}\label{eq:qq01}
q(a_0)-q(a) &=\int_0^s \xi \left(C_2 \xi^{\beta-\gamma-1}(1+o(1))\right)\, d\xi= 
\frac{\gamma}{\beta} C_2 s^{\beta-\gamma+1} (1+o(1)) \nonumber \\
&=\frac{\gamma}{\beta} C_2\left(\frac{a-a_0}{C_2}\right)^{\frac{\beta-\gamma+1}{\beta-\gamma}}(1+o(1)),
\end{align}
where in the second equality we have used~\eqref{aa000}. So, there is $c_2 > 0$ such that
\[
a-a_0=c_2 \left(q(a_0)-q(a)\right)^{\frac{\beta-\gamma}{\beta-\gamma+1}}(1+o(1)),
\]
For an arbitrary measure $\nu$ we have
$P_{\nu_\phi}(\phi)- P_\nu(\phi) \ge P_{\nu_\phi}(\phi)- P_{\nu_s}(\phi)$
as in~\eqref{eq:near_final}, we have
\[
 \int\psi\,d\nu-\int\psi\,d\nu_{\phi}\le c_2 (P_{\nu_\phi}(\phi)-P_\nu(\phi))^{\frac{\beta-\gamma}{\beta-\gamma+1}}
 \]
as required. For the equilibrium state $\nu_s$ itself, we have the more precise estimate with $c_2 = \frac{\beta}{\gamma}C_2$:
\[
 \int\psi\,d\nu_s-\int\psi\,d\nu_{\phi} = c_2 (P_{\nu_\phi}(\phi)-P_{\nu_s}(\phi))^{\frac{\beta-\gamma}{\beta-\gamma+1}} (1+o(1)),
 \]
which can be rewritten to~\eqref{eq:heavyHold}.\\

 \textbf{Proof of item (b), the case $\bm{\beta/\gamma\in (2,3)}$.}
 Using   Proposition~\ref{prop:gmsec}(ii) and Taylor's theorem, we have

\begin{align}\label{eq:aa0}
a-a_0&= p'(s) - p'(0) = s p''(0) + \int_0^s \xi p'''(\xi) \, d\xi
= s\sigma^2 + O(s^{\beta-2\gamma+1}).
\end{align}
Therefore
\begin{align}\label{aa0third}
s=\frac{a-a_0}{\sigma^2}\left(1+O(s^{\beta-2\gamma})\right).
\end{align}
By Taylor's theorem,
$p(s)=p(0)+ sp'(0)+\frac{s^2}{2}p''(0)+\int_0^s\xi^2 p'''(\xi)\, d\xi$.
This together with Proposition~\ref{prop:gmsec}(ii) (and recalling $p''(0)=\sigma^2$ and $p(0)=0$) gives
\begin{align*}
q(a_0)-q(a) &=sp'(s)-p(s)\\
&=sp'(s) - \left(p(0) + sp'(0) + \frac{s^2}{2} p''(0) + \int_0^s \xi^2 p'''(\xi) \, d\xi \right)\\
&=s(p'(s)-p'(0)) - \frac{s^2}{2} \sigma^2 -  \int_0^s \xi^2 p'''(\xi) \, d\xi\\
&=s^2\sigma^2 + O(s^{\beta-2\gamma +2}) - \frac{s^2}{2} \sigma^2 + O(s^{\beta-2\gamma +2})
= \frac{s^2}{2} \sigma^2 (1 + O(s^{\beta-2\gamma})),
\end{align*}
where we used~\eqref{eq:aa0} in the last line.
This together with~\eqref{aa0third} gives
\begin{align}\label{eq:qq02}
q(a_0)-q(a) &= \frac{(a-a_0)^2}{2\sigma^2} \left( 1+ O\left((a-a_0)^{\beta-2\gamma}\right) \right).
\end{align}
Since for an arbitrary measure $\nu$ we have again
\[
 \int\psi\,d\nu-\int\psi\,d\nu_{\phi}\le c_3 \sqrt{ P_{\nu_\phi}(\phi)-P_\nu(\phi)}
 \]
 for some $c_3 \geq 1$.
For the equilibrium state $\nu_s$ itself, we have the more precise estimate:
\[
 \int\psi\,d\nu_s-\int\psi\,d\nu_{\phi} =
\sigma \sqrt{2}  \sqrt{ P_{\nu_\phi}(\phi)-P_{\nu_s}(\phi) } \left( 1+O\left(\left(P_{\nu_\phi}(\phi)-P_{\nu_s}(\phi) \right)^{\frac{\beta-2\gamma}{2}}\right) \right).
 \]
This can be rewritten to~\eqref{eq:heavyquad}
 \end{pfof}

\subsection{When $a=\int\psi~d\nu$ is much larger than $\int\psi~d\nu_\phi$}

\begin{pfof}{Theorem~\ref{thm:bigint}}
 First notice that from (GM1) and Abramov's formula that $\int\psi~d\nu<C_0$, so  $C_{\phi, \psi}'':=\max\left\{\int\psi~d\nu_\phi, C_0\right\}$ and set $\psi':=\frac{\psi}{C_{\phi, \psi}''}$.  We will use $q= q_{\phi, \psi'}$ and (implicitly) $p=p_{\phi, \psi'}$ here.

We follow the proof of \cite[Theorem 7.1]{RuhSar22}.
The following is an analogue of \cite[Lemma 5.1]{RuhSar22}.

\begin{lemma}
\begin{itemize}
\item[(a)] $q= q_{\phi, \psi'}$  is well defined and finite on $\left(\int\psi'~d\nu_\phi, \sup_{\nu\in\M_F} \int\psi'~d\nu\right)$;
\item[(b)]
$q= q_{\phi, \psi}$ is concave on the domain on $\left(\int\psi~d\nu_\phi, \sup_{\nu\in\M_F} \int\psi'~d\nu\right)$. 
\end{itemize}
\label{lem:concave}
\end{lemma}

\begin{proof}
For part (a) we follow the proof of \cite[Lemma 5.1]{RuhSar22}, but since in general we do not have information on $p_{\phi, \psi'}(t)$ for $t<0$, or the topological entropy of $F$,  we start by assuming that $a\in \left(\int\psi'~d\nu_\phi, \sup_{\nu\in\M_F} \int\psi'~d\nu\right)$.  Note that the theory above (more precisely, the arguments use inside the proofs of Theorems~\ref{thm:firstmain} and~\ref{thm:secmain})  shows that $q$ is well defined in a subset of this set, but here we look to extend this.  The choice of $a$ implies there exist $\nu_1, \nu_2\in \M_F$ such that 
$$\int\psi'~d\nu_1<a<\int\psi'~d\nu_2,$$
so as in  \cite[Lemma 5.1]{RuhSar22}, $\int\psi'~d\nu=a$ for some convex combination of $\nu_1$ and $\nu_2$ and the supremum defining $q$ is over a non-empty set and it is well-defined.  The same argument pushed to the suspension flow, of  \cite[Lemma 5.1]{RuhSar22} implies that $q(a)>-\infty$.

Finally, the proof of part (b) is identical to the latter part of the proof of \cite[Lemma 5.1]{RuhSar22}.  
\end{proof}

For the next step we follow a slightly coarser version of the proof of \cite[Corollary 5.1(2)]{RuhSar22}.  The first step is to show that $q$ is strictly decreasing.  We note that the proofs of  Proposition~\ref{prop:gm} or \ref{prop:gmsec} imply that $p$ is analytic in some interval $(\eps_1, \eps_2)$ for $\eps_1, \eps_2>0$, where $\eps_1$ can be taken arbitrarily close to 0.  The same arguments as in \cite[Lemma 5.2]{RuhSar22}, see in particular (5.3), then also imply that $q$ is differentiable and strictly concave on some interval $(a_0', a_1)$ where $a_0'$ can be taken arbitrarily close to $a_0=\int\psi'~d\nu_\phi=p'(0)$, and moreover $q'(p'(t))=-t$ for $p'(t)\in (a_0', a_1)$.  The key fact we then take from this is that $q(a_1)<q(a_0)$ so we set $\eta:= \frac{q(a_0)-q(a_1)}{a_1-a_0}>0$.  Then since Lemma~\ref{lem:concave} implies that $q$ is concave for $a>a_0$, so for $a>a_1$ we have $q(a)-q(a_0)< -\eta (a-a_0)$.  

Given that $a=\int\psi'~d\nu$, as in the proof of Theorem~\ref{thm:firstmain} or \ref{thm:secmain}, the definition of $q$ implies $P_\nu(\phi)\le q(a)$ and hence we can interpret the inequality above as: if $a\ge a_1$ then 
\begin{equation}
q(a_0)-q(a)\ge \eta(a-a_0)\Longrightarrow \int\psi'~d\nu-\int\psi'~d\nu_\phi\le \frac1{\eta} (P(\phi)-P_\nu(\phi)).\label{eq:linapprox}
\end{equation}

Then following the argument of the proof of \cite[Theorem 7.1]{RuhSar22}.  From \eqref{eq:linapprox}, if $\int\psi~d\nu>a_1$ then
$$\frac12\left(\int\psi'~d\nu-\int\psi'~d\nu_\phi\right)\le \frac1{2\eta} (P(\phi)-P_\mu(\phi)).$$
Also then noticing that $\frac12\left(\int\psi'~d\nu-\int\psi'~d\nu_\phi\right)\le 1$ we trivially have 
$$\frac12\left(\int\psi'~d\nu-\int\psi'~d\nu_\phi\right)\le \left(\frac12\left(\int\psi'~d\nu-\int\psi'~d\nu_\phi\right)\right)^\rho$$ for any $\rho\in (0,1)$ (for example $\rho=1/2$).  Thus,
$$\int\psi~d\nu-\int\psi~d\nu_\phi\le \frac{2C_{\phi, \psi}''}{(2\eta)^\rho}\left(P(\phi)-P_\nu(\phi)\right)^\rho.$$

We set  $C_{\phi, \psi}'$ to be the maximum of $ \frac{2C_{\phi, \psi}''}{(2\eta)^\rho}$  and the constant coming from our main theorems.
\end{pfof}

\section{Applications} \label{sec:examples}

We provide examples of systems, both of discrete and continuous time, for which our main results  apply. 
These are systems with weak forms of hyperbolicity that have not been studied before form this point of view.

\subsection{Intermittent interval maps} \label{sec:int} Zweim\"uller \cite{zw} introduced a class of interval maps 
$f:[0,1] \to [0,1]$ that he called AFN maps, i.e., non-uniformly expanding maps with finitely many branches, 
finitely many neutral fixed points, and satisfying Adler's distortion property ($f''/f'^{2}$ bounded). 
Note that AFN maps are, in general, non-Markov. 
We stress that these are maps with weak hyperbolicity properties. Let $\alpha\in (0,1)$ and $b\in (0, 1]$ consider the 
family of AFN maps defined by
\begin{eqnarray*}
 f(x) = f_{\alpha, b}(x) = \begin{cases} x(1+ 2^\alpha x^\alpha) & \text{ if } x\in [0, 1/2],\\
b(2x-1) & \text{ if } x\in (1/2, 1].
\end{cases}
\end{eqnarray*}
It follows from \cite{zw} that for this range of values of the parameters $\alpha$ and $b$, 
there exists an absolutely continuous probability measure $\mu$. Moreover, the first return time map
to $Y= (1/2, 1]$ 
is uniformly expanding, although it may not be Markov. In \cite[Section 9]{BruTer18}, 
a Gibbs-Markov inducing scheme for $Y$ with return time $\tau$ is constructed. That is, there exists a countable partition 
of $Y$ so that $\tau$ is constant on each of the elements of the partition and the map $T:Y\to Y$ defined by $T=f^\tau$ is Gibbs-Markov.  
The map $T$ can be thought of as a discrete suspension of $f$ with roof function $\tau$. Moreover, for a potential $\psi:[0,1] \to \R$ 
its induced version $\bar \psi: Y \to \R$ is defined by $\bar \psi = \sum_{j=0}^{\tau -1} \psi \circ f^j$.
In particular, our main results can be applied to this discrete time system.  
We now verify that under certain conditions the assumptions of our results are indeed satisfied. We begin with Theorem~\ref{thm:firstmain}.

It was was established in \cite[Section 9]{BruTer18}  that for $\beta=1/\alpha$ there exists $c>0$ such that  the following bound on the tails holds,
\begin{equation*}
\mu_Y(\tau\ge n)\sim cn^{-\beta}.
\end{equation*}
That is, assumption (GM0) is fulfilled.

Note that if $\alpha \in (0, 1/2)$ then $\beta > 2$  and   if $\alpha \in (1/2,1)$ then $\beta \in (1,2)$.
 
Recall that (GM1) is an assumption on the induced version of a potential $\psi$. It states that there exists 
$\gamma  \in (\beta-1, \beta)$ such that $\bar \psi=C_0 - \psi_0$ with $0 \leq \psi_0 \leq C_1 \tau^{\gamma}$. 
The last assumption in  Theorem~\ref{thm:firstmain}, besides (GM0) and (GM1), is that $q_1>3$, which in particular implies that  
$\beta / \gamma >3$. Under the assumptions of (GM1) we  have that $\beta / \gamma \in (1, \beta / (\beta -1))$. Also, for $\beta >2$ we have $\beta / (\beta -1)<2$. Thus, if $\alpha \in (0, 1/2)$ then the assumptions  of Theorem~\ref{thm:firstmain} can not be satisfied  ($q_1$ is always smaller than 3). However, for $\alpha \in (1/2, 1)$ the result holds.

\begin{prop} \label{prop:1}
The conclusions of Theorem~\ref{thm:firstmain} hold for the induced system $(T, \mu_Y)$ with  $\alpha \in (1/2,1)$ and 
$\psi:[0,1] \to \R$ a H\"older function such that  $\psi(x) = -x^{(1-\gamma)\alpha}$ for 
$\gamma \in ((1-\alpha)/\alpha, \alpha /(\alpha +1))$, $\beta / \gamma >3$ and $x$ in a neighbourhood of $0$.

In the case $\beta>3$ we can consider the case $\gamma=1$ in this setting.  Here we can for example choose $\psi$ to be H\"older and negative (bounded below by $-C_1$) in $Y^c$ and to be equal to $C_0$ and Theorem~\ref{thm:firstmain} holds.
\end{prop}

\begin{proof}
We already established that assumption (GM0) is satisfied.  It was proved in \cite[Proposition 8.5]{BruTerTod19} that if $\gamma \in (0, \alpha / (\alpha +1))$ then the induced potential satisfies 
$\bar\psi(x) \sim C-\tau(x)^\gamma$ as $x\to 1/2$. Thus, the parameter $\gamma$ has to be chosen from the set 
$(\beta-1, \beta) \cap (0, \alpha / (\alpha +1))$ so as   $\beta / \gamma >3$. These conditions are compatible, 
so we can assume that $q_1>3$ and that (GM1) is fulfilled.

For the final statement note that in this setting $\bar\psi(x) = C_0-\psi_0(x)$ where $0\le \psi_0(x) \le C_1\tau(x)$.
 \end{proof}

Similarly, we obtain a version of  Theorem~\ref{thm:secmain}  in the same range of values of $\alpha$, but for a different range of values of $\gamma$.

\begin{prop} \label{prop:2}
The conclusions of Theorem~\ref{thm:secmain} hold for the induced system $(T, \mu_Y)$ with  $\alpha \in (1/2,1)$ and 
$\psi:[0,1] \to \R$ a function such that there exists $\gamma \in (\beta-1, 1)$ for which $\bar \psi = C_0-C_1 \tau^{\gamma}$. 
Both cases, $\beta /\gamma \in (1,2]$ and $\beta /\gamma \in (2,3)$, occur.
\end{prop}

In the case $b=1$ a construction to produce $\psi$ as above is given as follows.  Let $x_0=1$ and $x_n=f_L^{-n}(1/2)$, where $f_L$ is the left branch of $f$.  Then on the intervals $X_n:= (x_n, x_{n-1}]$ define $\psi|_{X_1} = C_0-C_1$ and $\psi|_{X_n} =C_1(-n^\gamma+(n-1)^\gamma)$, so for $x$ having $\tau(x) = n$, $\bar\psi = C_0+C_1\sum_{k=1}^n(-n^\gamma+(n-1)^\gamma) = C_0 - C_1n^\gamma$, as required.

Observe that for $\alpha \in (0, 1/2)$ we have $\beta >2$ and for Theorem~\ref{thm:secmain} to hold we require 
$\beta \in (1,2)$. Therefore, the appropriate range of values of $\alpha$ in order to apply our main results is $(1/2, 1)$.

\subsection{Suspensions over intermittent interval maps}

In this section we consider suspension flows over the induced map $T$ defined in Section~\ref{sec:int}. 
Essentially, this is a continuous time representation of $T$ that preserves its main properties. 
Let $\rho: Y \to \R^+$ be a H\"older function bounded away from zero. Let $\bar \tau: Y \to \R^+$ 
be defined by $\bar \tau (x)= \sum_{j=0}^{\tau(x)-1}
\rho(f^j x)$. Let $(\F_t)_t$ be the suspension (semi)flow with base map $T$ and roof function $\bar \tau$. 
Since $\rho$ is bounded, assumption (GM0) is satisfied (as in Section~\ref{sec:int}) for the measure $\mu_Y$. 

A standard tool to construct examples in suspension flows is the following. Given a regular potential defined 
on the base space $g: Y \to \R$, construct a  continuous potential $\psi: Y^{\bar \tau} \to \R$ so that its 
induced version coincides with $g$, that is $\bar \psi = g$. Details of this type of construction can be found 
in \cite{brw}, minor adaptations are required in this setting. Since the assumptions of our main results are 
in terms of the induced potentials, this tool allows us to state flow versions of Propositions~\ref{prop:1} and~\ref{prop:2}. 
Indeed, we just need to consider potentials $\psi: Y^{\bar \tau} \to \R$ so that its induced versions satisfy 
the properties of the induced potentials $\bar \psi$ in Propositions~\ref{prop:1} and~\ref{prop:2}.


\begin{thebibliography}{10}
\bibitem[A1]{Aaronson}
J.~Aaronson, \emph{{An Introduction to Infinite Ergodic Theory}}. Math.\ Surveys
  and Monographs \textbf{50}, Amer.\ Math.\ Soc., 1997.

\bibitem[AD]{AD01} J.~Aaronson, M.~Denker,
\emph{Local limit theorems for partial sums of stationary sequences generated by Gibbs-Markov maps.} 
{Stoch.\ Dyn.} \textbf{1} (2001) 193--237.

\bibitem[Ab]{ab} L.~Abramov,
\emph{On the entropy of a flow}.
Dokl.\ Akad.\ Nauk SSSR \textbf{128} (1959) 873--875.

\bibitem[AK] {AmbKak42} W.~Ambrose, S.~Kakutani,
\emph{Structure and continuity of measurable flows}.
Duke Math.\ J.\ \textbf{9} (1942) 25--42.

\bibitem[BRW]{brw} L.\ Barreira, L.\ Radu,  C.\ Wolf,
\emph{Dimension of measures for suspension flows.}
Dyn. Syst. \textbf{19} (2004) no. 2, 89--107.

\bibitem[BT]{BruTer18} H.\ Bruin, D.\ Terhesiu,
\emph{Upper and lower bounds for the correlation function via inducing with general return times.}
{Erg.\ Th.\ and Dyn.\ Syst.\ }  \textbf{38} (2018) 34--62.

\bibitem[BTT1]{BruTerTod19} H.\ Bruin, D.\ Terhesiu, M.\ Todd,
\emph{The pressure function for infinite equilibrium measures.}
Israel J.\ Math.\ \textbf{232} (2019) 775--826. 

\bibitem[BTT2]{BruTerTod21} H.\ Bruin, D.\ Terhesiu, M.\ Todd,
\emph{ Pressure function and limit theorems for almost Anosov flows.}
Comm.\ Math.\ Phys.\ \textbf{382} (2021) 1--47. 

\bibitem[CS]{CyrSar09} V.\ Cyr, O.\ Sarig, 
\emph{Spectral gap and transience for Ruelle operators on countable Markov shifts.} 
Comm.\ Math.\ Phys.\ \textbf{292} (2009) no. 3, 637--666.

\bibitem[EK]{ek} M. Einsiedler, S. Kadyrov, 
\emph{Entropy and escape of mass for $SL3(\Z) \setminus SL3(\R)$.} 
Israel J.\ Math.\ \textbf{190} (2012) 253--288.

\bibitem[G]{Gou04} S.\ Gou\"ezel,
\emph{Central limit theorem and stable laws for intermittent maps.}
Probab.\ Theory Related Fields \textbf{128} (2004) 82--122.
  
\bibitem[K]{Kad15}  S.\ Kadyrov,  \emph{Effective uniqueness of Parry measure and exceptional sets in ergodic theory.} Monatsh.\ Math.\  \textbf{78}(2) (2015), 237--249.

\bibitem[LSV]{LiveraniSaussolVaienti99}
C.~Liverani, B.~Saussol, S.~Vaienti.
\emph{A probabilistic approach to intermittency.}
{Erg.\ Th.\ and Dyn.\ Syst.\ } \textbf{19} (1999) 671--685.

\bibitem[MT]{MelTer17} I.\ Melbourne, D.\ Terhesiu. 
\emph{Operator renewal theory for continuous time dynamical systems with finite and infinite measure}. 
Monatsh.\ Math.\  \textbf{182} (2017) 377--431.

\bibitem[MTo]{MT04}
I.~Melbourne, A.~T{\" o}r{\" o}k,
\emph{Statistical limit theorems for suspension flows}.
Israel J.\ Math.\  \textbf{144} (2004) 191--209.

\bibitem[P]{Po11}  F.\ Polo, \emph{Equidistribution in chaotic dynamical systems.}
ProQuest LLC, Ann Arbor, MI, 2011. Thesis (Ph.D.) The Ohio State University.

\bibitem[PU]{PrzUrb10} F.\ Przytycki, M.\ Urba\'nski,
\emph{Conformal fractals: ergodic theory methods.}
London Mathematical Society Lecture Note Series, vol. 371, CUP, 2010.

\bibitem[R]{Ruh21} R.\ R\"uhr, \emph{Pressure inequalities for Gibbs measures of countable Markov shifts.} Dynamical Systems, \textbf{36}  (2021) no.2 332--339.  Correction: \emph{Pressure inequalities for Gibbs measures of countable Markov shifts.} Dyn. Syst. \textbf{37} (2022) no.2, 354--355.

\bibitem[RS]{RuhSar22} R.\ R\"uhr, O.\ Sarig, 
\emph{Intrinsic ergodicity for countable state Markov shifts.}
 Israel J.\ Math. \textbf{251} (2022) 679--735.

 \bibitem[S1]{Sar06}
 O.\ Sarig, \emph{Continuous phase transitions for dynamical systems.}
Comm.\ Math.\ Phys. \textbf{267} (2006) 631--667.

 \bibitem[S2]{Sar15}
 O.\ Sarig,
 \emph{Thermodynamic formalism for countable Markov shifts, Hyperbolic dynamics, fluctuations and large deviations.}
 Proc. Sympos. Pure Math., vol. 89, Amer. Math. Soc., Providence, RI, 2015.

\bibitem[Z]{zw} R.\ Zweim\"uller, \emph{Ergodic structure and invariant densities of non-Markovian interval maps with indifferent fixed points.} Nonlinearity \textbf{11} (1998) no. 5, 1263--1276.

\end{thebibliography}
\end{document}